\documentclass[preprint, 12pt]{elsarticle}
\usepackage[utf8]{inputenc}
\usepackage{amssymb}
\usepackage{amsthm}
\usepackage{mathrsfs}
\usepackage{amsfonts}
\usepackage{amsmath}
\usepackage{mathtools}
\usepackage{makecell}
\usepackage{lipsum}
\makeatletter
\def\ps@pprintTitle{%
 \let\@oddhead\@empty
 \let\@evenhead\@empty
 \def\@oddfoot{}%
 \let\@evenfoot\@oddfoot}
\makeatother

\newcommand{\ps}{\frac{\partial}{\partial s}}

\newcommand{\LL}{\mathcal{L}}
\newcommand{\scal}{\textnormal{scal}}
\newcommand{\Rm}{\textnormal{Rm}}
\newcommand{\Ric}{\textnormal{Ric}}

\newcommand{\aLL}{\overline{\mathcal{L}}}
\newcommand{\pt}{\frac{\partial}{\partial t}}
\newcommand{\psp}{\frac{\partial^+}{\partial s}}
\newcommand{\ptp}{\frac{\partial^+}{\partial t}}
\numberwithin{equation}{section}

\newtheorem{theorem}{Theorem}[section]
\newtheorem{lem}[theorem]{Lemma}
\newtheorem{remark}[theorem]{Remark}
\newtheorem{prop}[theorem]{Proposition}
\newtheorem{claim}[theorem]{Claim}
\newtheorem{cor}[theorem]{Corollary} 
\theoremstyle{definition}
\newtheorem{defn}[theorem]{Definition}

\usepackage{xpatch}
\makeatletter
\xpatchcmd{\tableofcontents}{\contentsname \@mkboth}{\small\contentsname \@mkboth}{}{}
\xpatchcmd{\listoffigures}{\chapter *{\listfigurename }}{\chapter *{\small\listfigurename }}{}{}
\makeatother

\begin{document}
\begin{frontmatter}

\title{Ricci flow under local almost non-negative curvature conditions }

\author[Yi]{Yi Lai\corref{cor}}
\ead{yilai@berkeley.math.edu}
\ead[url]{https://math.berkeley.edu/~yilai}
\cortext[cor]{Corresponding author.}
\address[berkeley]{Department of Mathematics, University of California, Berkeley, CA 94720, USA}

\begin{abstract}
We find a local solution to the Ricci flow equation under a negative lower bound for many known curvature conditions. The flow exists for a uniform amount of time, during which the curvature stays bounded below by a controllable negative number.
The curvature conditions we consider include 2-non-negative and weakly $\textnormal{PIC}_1$ cases, of which the results are new.
We complete the discussion of the almost preservation problem by Bamler-Cabezas-Rivas-Wilking, and the 2-non-negative case generalizes a result in 3D by Simon-Topping to higher dimensions.

As an application, we use the local Ricci flow to smooth a metric space which is the limit of a sequence of manifolds with the almost non-negative curvature conditions, and show that this limit space is bi-H$\ddot{\textnormal{o}}$lder homeomorphic to a smooth manifold. 

\end{abstract}

\end{frontmatter}

%%INEQUALITIES
%%there exists constant C such that \frac{1}{t}exp(-\frac{1}{t})\le C exp(-\frac{1}{C t})
%%\int\displaylimits_{R^n-B(0,d)}e^{-|y|^2}dy\le C(n)e^{-0.5d^2}

%%STRUCTURE
%%1.induction argument left blank of u(x,t)
%%2.heat kernel estimate
%%3.cut-off function
%%4.combine the results to fill in the blank of u(x,t)\le 1.
\setcounter{tocdepth}{1}
% * <yilai@berkeley.edu> 2018-05-21T10:47:13.447Z:
%
% ^.
%\tableofcontents

\begin{section}{Introduction and main results}
Ricci flow as introduced by Hamilton in \cite{hamilton1}, describes the evolution of a time-dependent family $g(t)_{\{t\in I\}}$ of Riemannian metrics on a manifold $M$:
% * <yilai@berkeley.edu> 2018-05-21T10:47:23.902Z:
%
% ^.
$$\frac{\partial}{\partial t}g(t)=-2\Ric(g(t)).$$
Here $\Ric(g(t))$ denotes the Ricci curvature of the metric $g(t)$. Hamilton used Ricci flow to prove that a compact three-manifold admitting a Riemannian metric of positive Ricci curvature must be a spherical space form. Since then Ricci flow has been used to prove many conjectures including the most remarkable Poincar\'e and Geometrization Conjectures in dimension $3$ by Perelman (\cite{Pel1}\cite{Pel2}\cite{Pel3}).

%As solving a PDE often requires a priori estimates on the solutions, a related problem in Rici flow consists in bounding from below the existence time of the flow, as well as controlling the evolution of some geometric quantities for positive time, in term of the conditions imposed on the initial metric. For example, as a classical result, Shi's theorem states that short-time existence holds true when the initial metric is complete and sectional curvature is bounded. It guarantees the existence of Ricci flow for a uniform amount of time, during which the curvature operator is uniformly controlled. Moreover, Many other curvature conditions have been studied for the existence problem and it often comes with the discussion about the preservation of initial curvature conditions. 

In general, Ricci flow tends to preserve some kind of positivity of curvatures. For example, positive scalar curvature is preserved in all dimensions. This follows from applying maximum principle to the evolution equation of scalar curvature, which is
\begin{equation*}
\pt \scal=\Delta \scal+2|\Ric|^2.
\end{equation*}
By developing a maximum principle for tensors, Hamilton \cite{hamilton1}\cite{hamilton2} proved that Ricci flow preserves the positivity of the Ricci tensor in dimension three and positivity of the curvature operator in all dimensions. 
H. Chen \cite{Chen} also proved the preservation of 2-non-negative curvature. The invariance of weakly $\textnormal{PIC}$ was first shown in dimension four by Hamilton \cite{PIC}, and the general case was obtained independently by Brendle and Schoen \cite{2} and by Nguyen \cite{17}. The curvature conditions weakly $\textnormal{PIC}_1$ and $\textnormal{PIC}_2$ were in turn introduced by Brendle and Schoen in \cite{2} and played a key role in their proof of the differentiable sphere theorem.
Finally in the K$\ddot{\textnormal{a}}$hler case, the condition of non-negative holomorphic bisectional curvature, which is a weaker condition than non-negative sectional curvature, is also preserved. This was shown by Bando \cite{Ba} in dimension three and by Mok \cite{Mo} in all dimensions. In \cite{Shi1997} Shi generalized this result to the complete K$\ddot{\textnormal{a}}$hler manifolds with bounded curvature.

In this paper, we study the preservation of almost non-negativity of curvature conditions. We say a quantity is almost non-negative when it has a negative lower bound. The almost non-negative case is less restrictive since it puts no constraints on the topology of the manifold. 
In \cite{almost}, Bamler, Cabezas-Rivas, and Wilking studied the complete manifold with bounded curvature, which satisfies global non-collapsedness and almost non-negativity for some curvature conditions. They showed that under the assumption, a Ricci flow exists for a uniform amount of time, during which the curvature can be bounded below by a negative constant depending only on initial conditions.
In the same paper, they also established some local results without the complete and curvature bound assumptions.  

However, the local cases of almost 2-non-negative curvature and weakly $\textnormal{PIC}_1$ remained unsolved. We verify these two local cases in this paper.
We use $\mathcal{C}$ to denote various non-negative curvature conditions, and write $\Rm\in \mathcal{C}$ to indicate that the curvature operator $\Rm$ satisfies the corresponding curvature condition. Then $\Rm+C\textnormal{I}\in\mathcal{C}$ indicates the nonnegativity of $\Rm+ C\textnormal{I}$, where $\textnormal{I}$ is the identity curvature operator whose scalar curvature is $n(n-1)$. Under this notation, our main theorem can be stated as below:

\begin{theorem}\label{t: theorem1}
Given $n\in\mathbb{N}$, $\alpha_0\in(0,1]$ and $v_0>0$, there exist positive constants $\tau=\tau(n,v_0,\alpha_0)$ and $C=C(n,v_0)$ such that the following holds: 
Let $(M^n,g_0)$ be a Riemannian manifold (not necessarily complete) and consider one of the following curvature conditions $\mathcal{C}$:
\begin{enumerate}
\item non-negative curvature operator;
\item 2-non-negative curvature operator\\
      (i.e. the sum of the lowest two eigenvalues of the curvature operator is non-negative);
\item weakly $\textnormal{PIC}_2$\\
      (i.e. non-negative complex sectional curvature, meaning that taking the cartesian product with $\mathbb{R}^2$ produces a non-negative isotropic curvature operator);
\item weakly $\textnormal{PIC}_1$\\
      (i.e. taking the cartesian product with $\mathbb{R}$ produces a non-negative isotropic curvature operator).
%\item non-negative bisectional curvature, in the case in which $(M,g)$ is K$\ddot{a}$hler with respect to some complex structure $J$.
\end{enumerate}

Suppose $B_{g_0}(x_0,s_0)\subset\subset M$ for some $x_0\in M$ and $s_0>4$ such that
\begin{equation}\label{condition}
\begin{dcases}
\textnormal{Rm}_g+\alpha_0\textnormal{I}\in\mathcal{C} \quad \textnormal{on} \;\;\; B_{g_0}(x_0,s_0)\;\\
Vol_{g_0}B_{g_0}(x,1)\ge v_0>0\quad \textnormal{for all}\;\;\; x\in B_{g_0}(x_0,s_0-1)\;.
\end{dcases}
\end{equation}
Then there exists a Ricci flow $g(t)$ defined for $t\in[0,\tau]$ on $B_{g_0}(x_0,s_0-2)$, with $g(0)=g_0$, such that for all $t\in[0,\tau]$, 
\begin{equation}\label{checking}
\begin{dcases}
|\textnormal{Rm}|_{g(t)}\le\frac{C}{t}\quad\quad\textnormal{on}\quad B_{g_0}(x_0,s_0-2)\\
\textnormal{Rm}_{g(t)}+C\alpha_0\textnormal{I}\in\mathcal{C}. 
\end{dcases}
\end{equation}

\end{theorem}

The results of the first and third conditions above were obtained in \cite{almost}. 
In dimensional three, 2-non-negative curvature has the same meaning as non-negative Ricci curvature, where the result was obtained by Simon and Topping in \cite{local3d} and \cite{mollification}.  

For each curvature condition $\mathcal{C}$, we define $\ell(x)\ge 0$ to be the smallest number such that $\Rm_{g}(x)+\ell(x)\,\textnormal{I}\in\mathcal{C}$. Then in each case the bound $\ell\le 1$ implies a lower bound on the Ricci curvature. We also observe that each curvature condition implies weakly $\textnormal{PIC}_1$. %In \cite{almost} the local results of the first and third conditions follow from the global results under a conformal change. The conformal change used there depends on the lower bound of sectional curvature, which is not guaranteed under the 2-non-negative and weakly $\textnormal{PIC}_1$ conditions. 
The method we use in the paper allows a uniform treatment of all curvature conditions that imply a lower bound for Ricci curvature and weakly $\textnormal{PIC}_1$. 
%\begin{theorem}\label{t: 1.2}
%Given $n\in\mathbb{N}$, $\alpha\in[0,1]$, $v_0>0$ and $\varepsilon\in(0,\frac{1}{10})$ there are constants $\tau=\tau(n,v_0,\alpha,\varepsilon)$ and $C=C(n,v_0,\varepsilon)$ such that the following holds.

%Let $\mathcal{C}$ be the curvature conditions listed in Theorem $\ref{t: theorem1}$. Suppose that $(M^n,g_0)$ is a Riemannian manifold (not necessarily complete), $x_0\in M$, $B_{g_0}(x_0,1)\subset\subset M$, with geometry controlled by
%\begin{equation}
%\begin{dcases}
%\textnormal{Rm}_g+\alpha\textnormal{I}\in\mathcal{C} \quad \textit{on} \quad B_{g_0}(x_0,1)\\
%Vol_{g_0}B_{g_0}(x_0,1)\ge v_0>0
%\end{dcases}
%\end{equation}
%Then there is a Ricci flow defined for $t\in[0,\tau]$ on the slightly smaller ball $B_{g_0}(x_0,1-\varepsilon)$ with $g(0)=g_0$ on $B_{g_0}(x_0,1-\varepsilon)$, such that for each $t\in[0,\tau]$, 
%\begin{equation}
%\begin{dcases}
%|\textnormal{Rm}|_{g(t)}\le\frac{C}{t}\quad\textit{on}\quad B_{g_0}(x_0,1-\varepsilon)\\
%\textnormal{Rm}_{g(t)}+C\alpha\textnormal{I}\in\mathcal{C}
%\end{dcases}
%\end{equation}

%\end{theorem}

As an application we have the following global existence result on complete manifolds with possibly unbounded curvature. It extends the corresponding results in \cite{almost} to the 2-non-negative and weakly $\textnormal{PIC}_1$ cases. 
%Bamler proved the global existence for some curvature conditions.

\begin{cor}\label{global existence}
Given $n\in \mathbb{N},\alpha_0\in(0,1]$ and $v_0>0$, there exist positive constants $C=C(n,v_0)$ and $\tau=\tau(n,v_0,\alpha_0)$ such that the following holds:
Let $\mathcal{C}$ be any curvature conditions listed in Theorem $\ref{t: theorem1}$, and $(M^n,g)$ be any complete Riemannian manifold (with possibly unbounded curvature) such that
\begin{equation}
\begin{dcases}
\Rm_g+\alpha_0\textnormal{I}\in\mathcal{C}\\
Vol_gB_g(p,1)\ge v_0 \quad\textnormal{for all} \,\,p\in M.
\end{dcases}
\end{equation}
Then there exists a complete Ricci flow $(M,g(t))_{t\in(0,\tau]}$ with $g(0)=g$ and so that
\begin{equation}
\begin{dcases}
\Rm_{g(t)}+C\alpha_0\textnormal{I}\in\mathcal{C}\quad\textnormal{for all}\,\,t\in(0,\tau]\,\,\textnormal{throughout}\,\,M\\
|\Rm|_{g(t)}\le\frac{C}{t}.
\end{dcases}
\end{equation}
\end{cor}

To prove the corollary we apply the local Ricci flow in Theorem $\ref{t: theorem1}$ to a sequence of larger and larger balls on the complete manifold. Thanks to the curvature decay estimate $|\Rm|\le\frac{C}{t}$ in $\eqref{checking}$, we can then take a convergent subsequence and get a globally defined flow.

Another application is the following smoothing result for singular limit spaces of sequences of manifolds with lower curvature bounds, which asserts the limit space is bi-H$\ddot{\textnormal{o}}$lder homeomorphic to a smooth manifold.

\begin{cor}\label{limit space}
Given $n\in\mathbb{N}$, $\alpha_0,v_0>0$. Let $\mathcal{C}$ be any curvature conditions listed in Theorem $\ref{t: theorem1}$, and
 $(M^n_i,g_i)$ be a sequence of complete Riemannian manifolds such that for all $i$, we have 
\begin{equation}
\begin{dcases}
\Rm_{g_i}+\alpha_0 \textnormal{I}\in\mathcal{C}\;\;\;\;\;\textnormal{throughout}\; M_i \\
Vol_{g_i}B_{g_i}(x,1)\ge v_0\;\;\;\;\;\textnormal{for all}\; x\in M_i
\end{dcases}
\end{equation}
Then there exists a smooth manifold $M$, a point $x_{\infty}\in M$, and a continuous distance metric $d_0$ on $M$ such that for some points $x_i\in M_i$, $(M_i,d_{g_i},x_i)$ converges in the pointed Gromov-Hausdorff distance sense to $(M,d_0,x_{\infty})$. Furthermore, the metric space $(M,d_0)$ is bi-H$\ddot{o}$lder homeomorphic to the smooth manifold $M$ equipped with any smooth metric.
\end{cor}

We give the proofs of Corollary $\ref{global existence}$ and $\ref{limit space}$ in Section 8.
We mention here that with some careful local distance distortion arguments, the same conclusion in Corollary $\ref{limit space}$ holds provided noncollapsedness of only one ball centered at a point. For detailed proof of this, we refer to \cite{mollification} where the argument is done for Ricci curvature and carries over to our case.

Finally, we sketch the proof of Theorem $\ref{t: theorem1}$ under some additional assumptions. That is, assuming $\eqref{condition}$ holds globally and a short time Ricci flow exists up to a uniform time $T<1$, during which $|\Rm|\le\frac{C}{t}$ holds, we want to show $\Rm_{g(t)}+C\alpha_0\textnormal{I}\in \mathcal{C}$ for all $t$. We define $\ell(x,t)$ by
\begin{equation}
\ell(x,t):=\inf\{\varepsilon\in[0,\infty)|\textnormal{Rm}_{g(t)}(x)+\varepsilon\textnormal{I}\in\mathcal{C}\}.
\end{equation}
Then it's equivalent to show $\ell(\cdot,t)\le C\alpha_0$ for all $t$. By \cite[Proposition~2.2]{almost}, $\ell$ satisfies an evolution inequality of the form
\begin{equation}\label{evolution}
 \pt \ell\le \Delta \ell+\scal\,\ell+C(n)\ell^2
\end{equation}
in the barrier sense for some dimensional constant $C(n)$. Assuming $\ell(x,t)\le 1$, then by the maximum principle, $\ell(\cdot,t)\le e^{C(n)t}h$ on $M\times[0,t)$, where $h$ solves
\begin{equation}
 \pt h=\Delta h+\scal \,h,\;\;\;h(\cdot,0)=\ell(\cdot,0).
\end{equation}
We can express this solution as
\begin{equation}\label{12}
 h(x,t)=\int_M G(x,t;y,0)\,\ell(y,0)\,d_0y,
\end{equation}
where $G(\cdot,\cdot;y,s)$ satisfies
\begin{equation}\label{11}
 (\pt-\Delta_{x,t}-\scal_{g(t)})G(x,t;y,s)=0\quad\textnormal{and}\quad \lim_{t\searrow s} G(x,t;y,s)=\delta_{y}(x).
\end{equation}
We say $G(\cdot,\cdot;y,s)$ is the heat kernel of equation $\eqref{11}$. It can be shown with the bound $|\Rm|_{g(t)}\le\frac{C}{t}$ that $G(x,t;y,s)$ has the following Gaussian upper bound
\begin{equation}
G(x,t;y,s)\le\frac{C}{(t-s)^{\frac{n}{2}}}exp\left(-\frac{d_s^2(x,y)}{C(t-s)}\right),
\end{equation}
substituting which into $\eqref{12}$ we get
\begin{equation}\label{int}
 \ell(x,t)\le e^{C(n)}h(x,t)\le \sup_{y\in M}\ell(y,0)\,\cdot\,\frac{C}{t^{\frac{n}{2}}}\int_{M} exp\left(-\frac{d_0^2(x,y)}{Ct}\right)\,d_0y\le C\,\sup_{y\in M} \ell(y,0).
\end{equation}

To prove Theorem $\ref{t: theorem1}$ by adapting the above argument, we need to overcome the difficulties caused by the lack of those additional assumptions. 
To construct a local Ricci flow, we use an extension method which was introduced in \cite{conformal} and \cite{mollification}.
The process starts by doing a conformal change to the initial metric, making it a complete metric and leaving it unchanged on $B_{g(0)}(x_0,r_1)$ for some $0<r_1<r_0=s_0$. 
%by a conformal change, keeping the metric unchanged in a smaller $r_1$-ball, and pushing the boundary to infinity. 
Then by the following doubling time estimate of Shi in \cite{Shi}, we can then run a complete Ricci flow up to a short time $t_1$.  
%(see for example \cite[Lemma~6.1]{doubling}), which guarantees the existence of a complete Ricci flow starting from a complete metric with bounded curvature, 
\begin{lem}{(Doubling time estimate)}\label{doubling}
Let $(M^n,g(0))$ be a complete manifold with bounded curvature $|\Rm|_{g(0)}\le K$, then there exits a complete Ricci flow $(M^n,g(t))$ such that
\begin{equation}
|\Rm|_{g(t)}\le 2K
\end{equation}
for all $0\le t\le\frac{1}{16K}$.
\end{lem}
Of course $t_1$ is uncontrolled and may depend on specific manifold due to the lack of a uniform curvature bound.
Next we do another conformal change to complete the metric at $t_1$, leaving it unchanged on $B_{g(0)}(x_0,r_2)$ for some $0<r_2<r_1$. 
Then using the doubling time estimate again, we have another complete Ricci flow from $t_1$ to $t_2$. 
Repeating the process, we obtain some successive complete Ricci flow pieces $(\{M_i\}_{i=1}^{n},\{g_i(t)\}_{i=1}^n)$, with each $M_i$ containing $B_{g(0)}(x_0,r_i)$. Restricting all the $g_i(t)$ on $B_{g(0)}(x_0,r_n)$, we thus obtain a smooth local Ricci flow $g(t)$ defined for all $t\in[0,t_n]$. The inductive construction is carried out in Section 6. 
%We need to verify condition $\eqref{checking}$ holds true after each inductive step, which guarantees the construction of next step. 

In particular, the curvature decay $|\Rm|_{g(t)}\le\frac{C}{t}$ in $\eqref{checking}$ together with the doubling time estimate enable us to choose $t_{i+1}=t_i(1+\frac{1}{16C})$ for each $i$. To verify $|\Rm|_{g(t)}\le\frac{C}{t}$ after each extension step, we use the curvature decay lemma in Section 3, which ensures the existence of $C$ under the assumption of a local upper bound of $\ell(\cdot,t)$. 
%We will see in the proof that the decreasing sequence $\{r_i\}$ has a uniform lower bound.

For the verification of $\ell(\cdot,t)\le C\alpha_0$ in $\eqref{checking}$, we perform a new local integration estimate, in which we use a generalized heat kernel. We know the standard heat kernel $G(x,t;y,s)$ on a complete Ricci flow satisfies the following reproduction formula for all $\mu<s<t$
\begin{equation}\label{reproduction}
\int G(x,t;y,s)\,G(y,s;z,\mu)\, d_sy=G(x,t;z,\mu).
\end{equation}
The standard heat kernel $G(x,t;y,s)$ is well defined by equation $\eqref{11}$ for all $(x,t)$ and $(y,s)$ in a same complete Ricci flow piece $(M_i,g_i(t))$ coming from the above inductive construction. In section 5, we use equation $\eqref{reproduction}$ inductively to make sense of $G(x,t;y,s)$ for $(x,t)$ and $(y,s)$ in different pieces and thus obtain a generalized heat kernel whose definition domain is on the whole $(\{M_i\}_{i=1}^{n},\{g_i(t)\}_{i=1}^n)$ and has a Gaussian upper bound.

%In section 6 we describe the inductive extension step. After each extension, we need to ensure that the conditions are still good enough to start the next extension. One of the conditions is $|\Rm|_{g(t)}\le\frac{C}{t}$, which we verify by a Curvature Decay Lemma (Lemma $\ref{l: curvature decay}$) from section 3. We verify the other condition $\ell\le C\,\alpha_0$ in section 7. 

\end{section}
%%%%%%%%%%%%%%%%%%%%%%%%%%%%%%%%%%%%%%%%%%%%%%%%%%%%%%%%%%%%%%%%%%%%%%%%%

%%%%%%%%%%%%%%%%%%%%%%%%%%%%introduce l and the inequality
\begin{section}{Preliminaries}
\setcounter{tocdepth}{1}
%\begin{subsection}{\textnormal{Evolution inequality of curvature quantities}}

%\begin{tabular}[none]{|l|l|l|}
%\hline
%$\textbf{Choices of S=S}_i$& $\mathcal{C}(\textbf{S}_i)=\{\cdots \textbf{curvature operators}\}$\\
%\hline
%$S_1=\mathfrak{so}(n,\mathcal{C})$& non-negative\\
%\hline
%$S_2=\{v\in\mathfrak{so}(n,\mathbb{C})\,|\,tr(v^2)=0\}$& 2-non-negative\\
%\hline
%$S_3=\{v\in \mathfrak{so}(n,\mathbb{C})\,|\,rank(v)=2,\,v^2=0\}$& weakly positive isotropic (\textnormal{PIC})\\
%\hline
%$S_4=\{v\in \mathfrak{so}(n,\mathbb{C})\,|\,rank(v)=2,\,v^3=0\}$& weakly $\textnormal{PIC}_1$\\
%\hline
%$S_5=\{v\in \mathfrak{so}(n,\mathbb{C})\,|\,rank(v)=2\}$& weakly $\textnormal{PIC}_2$/non-negative complex\\
%\hline
%\end{tabular}

%For any curvature condition $\mathcal{C}$, we consider the smallest $\ell\ge 0$ for which $\Rm+\ell \textnormal{I}\in\mathcal{C}$.  That is, for a Ricci flow $(M^n,g(t))$, we define
%\begin{equation}
%\ell(p,t):=\inf\{\alpha_0\in[0,\infty)|\textnormal{Rm}_{g(t)}(p)+\alpha_0\textnormal{I}\in\mathcal{C}\}
%\end{equation}

%It is shown in \cite{almost} that for all the curvature conditions listed in Theorem $\ref{t: theorem1}$, there is a constant $C$ depending only on dimension such that $\ell$ satisfies the following inequality
%\begin{equation}\label{e: evolution of l}
%\partial_{t}\ell\le\Delta\ell+\textnormal{scal}\ell+C\ell^2
%\end{equation}

%\end{subsection}
\begin{subsection}{\textnormal{Local distance distortion estimates}}
%%%%%%%%%%%%%%%%%%%%%%%%%%%%%%

We need the following distance distortion estimates, which are originally due to Hamilton \cite{distanceH} and Perelman \cite{Pel1}, and phrased and improved in \cite{mollification}. These estimates ensure the distance between two points won't expand or shrink too soon when assuming $\Ric\ge -K$ or $\Ric\le\frac{C}{t}$, respectively.

\begin{lem}{(Expanding Lemma)}.\label{l: expanding}
Given $T,K,R>0$ and $n\in \mathbb{N}$. Let $(M^n,g(t))$ be a Ricci flow for $t\in[-T,0]$. Suppose for some $x_0\in M$ we have $B_{g(0)}(x_0,R)\subset\subset M$ and $\textnormal{Ric}_{g(t)}\ge -K$ on $B_{g(0)}(x_0,R)\cap B_{g(t)}(x_0,Re^{Kt})$ for each $t\in[-T,0]$. 

Then for all $t\in[-T,0]$,
\begin{equation}
B_{g(0)}(x_0,R)\supset B_{g(t)}(x_0,Re^{Kt}),
\end{equation}
or equivalently, for all $y\in B_{g(0)}(x_0,Re^{Kt})$ we have
\begin{equation}
d_{g(t)}(y,x_0)\ge d_{g(0)}(y,x_0)e^{Kt}.
\end{equation}
\end{lem}

\begin{lem}{(Shrinking Lemma)}. \label{l: shrinking} Given $T,c_0,r>0$ and $n\in\mathbb{N}$, there exists constant $\beta=\beta(n)\ge 1$ such that the following holds: Let $(M^n,g(t))$ be a Ricci flow for $t\in [0,T]$. Suppose for some $x_0\in M$ we have $B_{g(0)}(x_0,r)\subset\subset M$. Suppose also $|\textnormal{Rm}|_{g(t)}\le\frac{c_0}{t}$, or more generally $\textnormal{Ric}_{g(t)}\le\frac{(n-1)c_0}{t}$, on $B_{g(0)}(x_0,r)\cap B_{g(t)}(x_0,r-\beta\sqrt{c_0 t})$ for each $t\in[0,T]$. 

Then for all $t\in[0,T]$, we have
\begin{equation}
B_{g(0)}(x_0,r)\supset B_{g(t)}(x_0,r-\beta\sqrt{c_0 t}),
\end{equation}
or equivalently, for all $y\in B_{g(t)}(x_0,r-\beta\sqrt{c_0t})$ we have
\begin{equation}
d_{g(t)}(y,x_0)\ge d_{g(0)}(y,x_0)-\beta\sqrt{c_0t}.
\end{equation}
More generally, for $0\le s\le t\le T$, we have
\begin{equation}
B_{g(s)}(x_0,r-\beta\sqrt{c_0 s})\supset B_{g(t)}(x_0,r-\beta\sqrt{c_0 t}),
\end{equation}
or equivalently, for all $y\in B_{g(t)}(x_0,r-\beta\sqrt{c_0t})$ we have
\begin{equation}
d_{g(t)}(y,x_0)\ge d_{g(s)}(y,x_0)-\beta(\sqrt{c_0t}-\sqrt{c_0s}).
\end{equation}
\end{lem}

As an application of the Shrinking Lemma, we get the following H$\ddot{\textnormal{o}}$lder estimate. 
\begin{lem}\label{holderinequality}
Given $T,c_0,r>0$ and $n\in\mathbb{N}$, there exist positive constants $\beta=\beta(n)$ and $\gamma=\gamma(c_0,n,T)$ such that the following holds: Let $(M^n,g(t))$ be a Ricci flow for $t\in[0,T]$, not necessarily complete. Suppose for some $x_0\in M$, we have $B_{g(t)}(x_0,2r)\subset\subset M$ for all $t\in[0,T]$. Suppose also $|\Rm|_{g(t)}(x)\le\frac{c_0}{t}$, or more generally $\Ric_{g(t)}(x)\le\frac{(n-1)c_0}{t}$ for all $x\in B_{g(t)}(x_0,2r)$ and $t\in[0,T]$. 

Then for all $x,y\in \,\bigcap_{s\in[0,T]} B_{g(s)}(x_0,r)$,
and $0\le t_1<t_2\le T$, we have
\begin{equation}\label{shrinking}
 d_{g(t_2)}(x,y)\ge d_{g(t_1)}(x,y)-\beta\sqrt{c_0}(\sqrt{t_2}-\sqrt{t_1}),
\end{equation}
Moreover, for all $t\in[0,T]$, we have
\begin{equation}\label{holder}
 d_{g(t)}(x,y)\ge \gamma[d_{g(0)}(x,y)]^{1+2(n-1)c_0}.
\end{equation}

\end{lem}
 
\begin{remark}
We need the curvature assumption on $B_{g(t)}(x_0,2r)\subset\subset M$ for all $t$ to estimate the distances change on $\cap_{s\in[0,T]}B_{g(s)}(x_0,r)$. The reason is that there are two ways to make sense of the distance at time $t$ between two points $x,y\in B_{g(t)}(x_0,2r)$. One is the infimum length of all connecting paths in $M$, and the other is the infimum length of all connecting paths that are contained in $B_{g(t)}(x_0,2r)$. The former is usually shorter than the latter.
These two metrics agree for $x,y\in B_{g(t)}(x_0,r)$ when $B_{g(t)}(x_0,2r)$ is compactly contained in $M$, and the distance can be realized by a geodesic which lies within $B_{g(t)}(x_0,2r)$.
\end{remark}

\begin{remark}
We can also prove the same conclusion for the Ricci flow defined only for $t\in(0,T]$, where $d_{g(0)}$ in $\eqref{holder}$ is replaced by the limit distance of $d_{g(t)}$. The limit exists thanks to bound $|\Rm|_{g(t)}\le\frac{C}{t}$ in $\eqref{checking}$.
\end{remark}

\begin{proof}[Proof of Lemma \ref{holderinequality}]
We note that there is no ambiguity to talk about $d_{g(t)}(x,y)$ for $x,y\in\,\bigcap_{s\in(0,T]} B_{g(s)}(x_0,r)$ for all $t\in[0,T]$, because the minimizing geodesic joining $x$ and $y$ with respect to $g(t)$ is contained in $B_{g(t)}(x_0,2r)\subset\subset M$. Inequality $\eqref{shrinking}$ follows by the above Shrinking Lemma. The proof of $\eqref{holder}$ follows by splitting $[0,t]$ into two intervals. We choose and fix $t_0=\frac{1}{c_0}\left[\frac{1}{2\beta}d_{g(0)(x,y)}\right]^2$. Then in the first interval $[0,t_0]$, we integrate the following inequality from Hamilton and Perelman
\begin{equation}
\ptp d_{g(t)}(x,y)\ge -\frac{\beta}{2}\sqrt{\frac{c_0}{t}}
\end{equation}
to get 
\begin{equation}\label{1}
d_{g(t_0)}(x,y)\ge \frac{1}{2}d_{g(0)}(x,y).
\end{equation}
By $\ptp|_{t_0} F$ we mean $\limsup_{t\rightarrow t_0^+}\frac{F(t)-F(t_0)}{t-t_0}$. In the second interval we use the following inequality, which follows from the Ricci flow equation
\begin{equation}
\ptp d_{g(t)}(x,y)\ge-(n-1)\frac{c_0}{t}d_{g(t)}(x,y),
\end{equation}
integrating which we get 
\begin{equation}\label{2}
d_{g(t)}(x,y)\ge d_{g(t_0)}(x,y)\left[\frac{t}{t_0}\right]^{-(n-1)c_0}.
\end{equation}
The combination of $\eqref{1}$ and $\eqref{2}$ gives $\eqref{holder}$. 
\end{proof}
\end{subsection}

\begin{subsection}{\textnormal{Extension Lemma}}
For the metric on a local region, we can modify it by a conformal change that pushes the boundary of the region, on which we have curvature bounds, to infinity in such a way that the modified metric is complete and has bounded curvature. 
For example, the open Euclidean unit ball can be made into a complete hyperbolic metric under a conformal change.
The following conformal change has been used in \cite{conformal}, \cite{mollification}. In \cite{almost}, a different conformal change was also used to achieve the local results of the first and third cases listed in Theorem $\ref{t: theorem1}$, as a corollary of their corresponding global results. 
\begin{lem}(Conformal Change Lemma)\label{l: conformal}
Let $(N^n,g)$ be a smooth (not necessarily complete) Riemannian manifold and let $U\subset N$ be an open set. Assume that for some $\rho\in(0,1]$, we have $\sup_U |\textnormal{Rm}|_g\le\rho^{-2}$, $B_g(x,\rho)\subset\subset N$ and $\textnormal{inj}_g(x)\ge\rho$ for all $x\in U$. Then there exist a constant $\gamma=\gamma(n)\ge 1$, an open set $\tilde{U}\subset U$ and a smooth metric $\tilde{g}$ defined on $\tilde{U}$ such that each connected component of $(\tilde{U},\tilde{g})$ is a complete Riemannian manifold satisfying
 \begin{enumerate}
 \item $|\textnormal{Rm}|_{\tilde{g}}\le\gamma\rho^{-2}$ and $\textnormal{inj}_{\tilde{g}}\ge\frac{1}{\sqrt{\gamma}}\rho$ for $x\in\tilde{U}$
 \item $U_\rho \subset\tilde{U}\subset U$
 \item $\tilde{g}=g$ on $\tilde{U}_\rho \supset U_{2\rho}$,
 \end{enumerate}
 where $U_s=\{x\in U|B_g(x,s)\subset\subset U\}$.
\end{lem}

\end{subsection}

\begin{subsection}{\textnormal{Some integrations}}
For later convenience, we include some frequently used inequalities and their proofs in this subsection.
\begin{lem}\label{calculation: whole integral}
Given $K,R,C_1>0$, $t\in(0,1]$ and $n\in\mathbb{N}$. There exists positive constant $C=C(K,C_1,n)$ such that the following holds. Let $(M,g)$ be a complete Riemannian manifold with $\textnormal{Ric}\ge-(n-1)K$ on $B_g(x,R)$ for some point $x\in M$. Then 
\begin{equation}
\frac{C_1}{t^{\frac{n}{2}}}\int\displaylimits_{B_{g}(x,R)} exp\left(-\frac{d^2_g(x,y)}{C_1 t}\right) d_{g}y\le C
\end{equation}
\end{lem}

\begin{proof}
Let $\hat{g}=\frac{1}{t}g$, then it suffices to show $I:=C_1\int\displaylimits_{B_{\hat{g}}(x,\frac{R}{\sqrt{t}})} exp(-\frac{d^2_{\hat{g}}(x,y)}{C_1})d_{\hat{g}}y\le C(C_1,K,n)$.
For all $y\in B_{\hat{g}}(x,\frac{R}{\sqrt{t}})$, the minimizing geodesic connecting $x$ and $y$ lies within $B_{\hat{g}}(x,\frac{R}{\sqrt{t}})$ where $\Ric\ge -Kt\ge -K$. So by Laplacian comparison the volume form $d_{\hat{g}}y\le sn_{-K}^{n-1}(r(y))dr\wedge dvol_{n-1}\le \frac{exp((n-1)\sqrt{K}r)}{(2\sqrt{K})^{n-1}}dr\wedge dvol_{n-1}$, where $r$ is the distance function centered at $x$ and $dvol_{n-1}$ is the standard volume form on $S^{n-1}(1)$. So we can express the integral on the segment domain in $T_xM$ and obtain
\begin{equation*}
\begin{split}
I&\le \frac{C_1}{(2\sqrt{K})^{n-1}}\int\displaylimits_{r\le\frac{R}{\sqrt{t}}} exp\left(-\frac{r^2}{C_1}\right)exp((n-1)\sqrt{K}r)\,dr\wedge dvol_{n-1}\\
&\le C(C_1,n,K)\int_{\mathbb{R}} exp\left(-\frac{r^2}{C_1}+(n-1)\sqrt{K}r\right)\,dr\le C(C_1,n,K)
\end{split}
\end{equation*}
\end{proof}

\begin{lem}\label{calculation: partial integral}
Given $C_1,C_2>0$ and $n\in\mathbb{N}$. Let $(M,g(t)),t\in[0,1]$ be a complete Ricci flow with $|\textnormal{Rm}|_{g(t)}\le\frac{C_1}{t}$. Then for any $d\ge 2(n-1)^{\frac{3}{2}}\sqrt{C_1}C_2$, 
\begin{equation}
\frac{C_2}{t^{\frac{n}{2}}}\int\displaylimits_{M-B_{g(t)}(x,\sqrt[4]{t}\,d)} exp\left(-\frac{d_t^2(x,y)}{C_2 t}\right)\,d_ty\le C\,exp\left(-\frac{d^2}{C\sqrt{t}}\right)
\end{equation}
where $C$ is a constant depending on $n$, $C_1$ and $C_2$.
\end{lem}

\begin{proof}
For convenience, $C$ denotes all the constants depending on $C_1$, $C_2$, and $C_3$. Fix $t$, let $\hat{g}=\frac{1}{t}g(t)$. Then it suffices to show 
\begin{equation}
C_2\int\displaylimits_{M-B_{\hat{g}}(x,\frac{d}{\sqrt[4]{t}})} exp\left(-\frac{d_{\hat{g}}(x,y)}{C_2}\right) d_{\hat{g}}y\le C\,exp\left(-\frac{d^2}{C\sqrt{t}}\right)
\end{equation}
with $|\textnormal{Rm}|_{\hat{g}}\le C_1$.

Since $\textnormal{Ric}\ge -(n-1)C_1$, we get by Laplacian comparison that the volume form $d_{\hat{g}}y\le sn_{-C_1}^{n-1}(r(y))dr\wedge dvol_{n-1}\le \frac{e^{(n-1)\sqrt{C_1}r}}{(2\sqrt{C_1})^{n-1}}dr\wedge dvol_{n-1}$ Thus by considering the integral over the segment domain in $T_xM$, denoting by $\omega_{n-1}$ the volume of $S^{n-1}(1)$, we get
\begin{equation*}
\begin{split}
I&\le C_2\int\displaylimits_{r\ge\frac{d}{\sqrt[4]{t}}} exp\left(-\frac{r^2}{C_2}\right)\,exp((n-1)\sqrt{C_1}r)\,dr\wedge dvol_{n-1}\\
&= C_2\,\omega_{n-1}\int\displaylimits_{r\ge\frac{d}{\sqrt[4]{t}}} exp\left(-\frac{r^2}{C_2}+(n-1)\sqrt{C_1}r\right)\,dr\\
%&\le C_2\,\omega_{n-1}\int\displaylimits_{r\ge\frac{d}{\sqrt[4]{t}}} exp(-\frac{r^2}{2C_2})\,dr\\
&\le C\int\displaylimits_{r\ge\frac{d}{\sqrt[4]{t}}} r\,exp\left(-\frac{r^2}{2C_2}\right)\,dr=C\,exp\left(-\frac{d^2}{2C_2\sqrt{t}}\right).
\end{split}
\end{equation*}
\end{proof}

\begin{lem}\label{calculation: inequity}
Given $t,T,d, C>0$ and $n\in\mathbb{N}$ such that $t<T\le d^2$, there exists positive constant $C_1=C_1(C,n)$ such that
\begin{equation}\label{15}
\frac{C}{t^{\frac{n}{2}}}exp\left(-\frac{d^2}{Ct}\right)\le \frac{C_1}{T^{\frac{n}{2}}}exp\left(-\frac{d^2}{C_1T}\right). 
\end{equation}

\begin{proof}
It's easy to see there exists $C_1=C_1(C,n)$ such that for all $x\in\mathbb{R}$,
\begin{equation}
\frac{1}{x^{\frac{n}{2}}}exp\left(-\frac{1}{C x}\right)\le C_1exp\left(-\frac{1}{2C x}\right).
\end{equation}
Then $\eqref{15}$ follows immediately from this inequality and the above assumptions.
%where $\tilde{C}=\tilde{C}(C,n)$, using which we have
%\begin{equation}
%\begin{split}
%\frac{C}{t^{\frac{n}{2}}}exp(-\frac{d^2}{Ct})&=\frac{C}{(\frac{t}{d^2})^{\frac{n}{2}}}\frac{1}{d^n}exp(-\frac{1}{C\frac{t}{d^2}})
%\le\frac{\tilde{C}}{d^n}exp(-\frac{1}{\tilde{C}\frac{t}{d^2}})\\
%&\le\frac{\tilde{C}}{T^{\frac{n}{2}}}exp(-\frac{d^2}{\tilde{C}t})
%\le\frac{\tilde{C}}{T^{\frac{n}{2}}}exp(-\frac{d^2}{\tilde{C}T})
%\end{split}
%\end{equation}
\end{proof}
\end{lem}

\end{subsection}

\begin{subsection}{\textnormal{Weak derivatives}}
Let $(M^n,g(t))$ be a Ricci flow, as we mentioned in introduction, $\ell$ satisfies the evolution inequality $\eqref{evolution}$ in the barrier sense: for any $(q,\tau)\in M\times(0,T)$ we find a neighborhood $\mathcal{U}\subset M\times (0,T)$ of $(q,\tau)$ and a $C^{\infty}$ function $\phi: \mathcal{U}\rightarrow \mathbb{R}$ such that $\phi\le\ell$ on $\mathcal{U}$, with equality at $(q,\tau)$ and
\begin{equation}
(\pt-\Delta)\phi\le\scal\ell+C(n)\ell^2\quad\textnormal{at}\;\;\;(q,\tau).
\end{equation}

%Let $(M^n,g(t))_{[a,b]}$ be a complete Ricci flow with bounded curvature for each $t$. 
Set $\LL=e^{-C(n)t}\ell$ and assume $\ell\le 1$ then by $\eqref{evolution}$ we have the following inequality which holds in the barrier sense
 \begin{equation}\label{star}
 (\frac{\partial}{\partial t}-\Delta)\mathcal{L}\le\textnormal{scal}\,\mathcal{L}.
 \end{equation}
Suppose for a moment that $\LL$ is smooth and $\psi(x,t)$ is a non-negative smooth function which is compactly supported in $M$ for each $t$. Then we see from the integration by parts formula that
\begin{equation}\label{by}
\begin{split}
 \pt \int_{U} \mathcal{L}\psi\,d_t x&=\int_U (\pt\mathcal{L}\psi-\mathcal{L}\psi\,\scal+\mathcal{L}\pt\psi)\,d_t x\\
 &\le\int_U ((\Delta\mathcal{L})\psi+\mathcal{L}\pt\psi)\,d_t x\\
 &=\int_U \mathcal{L}(\Delta\psi+\pt\psi)\,d_t x.
 \end{split}
\end{equation}

We show in Lemma $\ref{smooth of cut-off}$ that some variant of $\eqref{by}$ is still true without the smooth assumptions either $\ell$ or the test function $\psi$. 

First, we give the definitions of inequalities in several weak senses. We say a continuous function $f:M\rightarrow \mathbb{R}$ satisfies $\Delta f\le u$ for some function $u:M\rightarrow \mathbb{R}$ in the barrier sense if for any point $x$ and every $\varepsilon>0$ there exists a neighborhood $\mathcal{U}_{\varepsilon}\subset M$ of $x$ and a smooth function $h_{\varepsilon}:\mathcal{U}_{\varepsilon}\rightarrow \mathbb{R}$ such that $h_{\varepsilon}(x)=f(x)$, $h_{\varepsilon}\ge f$ in $\mathcal{U}_{\varepsilon}$ and $\Delta h_{\varepsilon}(x)\le u(x)+\varepsilon$. 

We say a continuous function $f:M\rightarrow \mathbb{R}$ satisfies $\Delta f\le u$ for some bounded function $u:M\rightarrow \mathbb{R}$ in the distributional sense if for any non-negative smooth function $h$ with compact support that $\int f\Delta h\le \int uh$. By standard argument, if $f$ satisfies $\Delta f\le u$ in the barrier sense, then $f$ satisfies it in the distributional sense (see for example \cite[Appendix~A]{appendix}).

%We say a continuous function $f:(a,b)\rightarrow\mathbb{R}$ satisfies $\pt f\le u$ for some function $u:(a,b)\rightarrow\mathbb{R}$ in the barrier sense if for any $t\in(a,b)$ and every $\varepsilon>0$ there exist $\delta>0$ and a smooth function $h_{\varepsilon}:[t,t+\delta)\rightarrow \mathbb{R}$ such that $h_{\varepsilon}\ge f$ in $[t,t+\delta)$ with equality at $t$, and $\pt h_{\varepsilon}(x)\le u(x)+\varepsilon$. 

\begin{lem}\label{g smooth}
Let $\psi(x,t)$ be a non-negative smooth function which is compactly supported in $M$ for each $t$. $\LL=e^{C(n)t}\ell$ with $\ell\le 1$. Then we have
\begin{equation}\label{banana}
\ptp\int \mathcal{L}\psi\,d_t x\le \int \mathcal{L}(\Delta \psi +\frac{\partial}{\partial t}\psi)\,d_t x
\end{equation}
for all $t\in[a,b)$, integrating which we have:
\begin{equation}\label{apple}
\left.(\int\LL \psi\,d_t x)\right|_{a}^{b}\le\int_{a}^{b}(\int\LL(\Delta\psi+\pt\psi))\,d_t x\,dt
\end{equation}
\end{lem}

\begin{proof}
Let $t_0$ be an arbitrary time in $[a,b)$. Since $\LL$ satisfies $$(\frac{\partial}{\partial t}-\Delta)\mathcal{L}\le\textnormal{scal}\,\mathcal{L}$$
in the barrier sense, by the maximum principle for complete manifold with bounded curvature, $\LL(\cdot,t)\le\overline{\LL}(\cdot,t)$ for all $t\in[t_0,b]$, where $\overline{\LL}$ is the solution to the initial value problem:
 \begin{equation}
(\ps-\Delta)\overline{\mathcal{L}}=\textnormal{scal}\overline{\mathcal{L}},\quad
\overline{\mathcal{L}}(\cdot,t_0)=\mathcal{L}(\cdot,t_0).
\end{equation}
Then $\aLL$ is smooth for all $t>t_0$ and so we have
\begin{equation}\label{tree}
\begin{split}
\left.\ptp\right|_{t_0}\int \mathcal{L}\psi\,d_t x&\le \left.\ptp\right|_{t_0}\int \aLL\psi\,d_t x=\lim_{t\rightarrow t_0^+}\pt\int\aLL\psi\,d_tx.\\
\end{split}
\end{equation}
For each $t>t_0$, we calculate by integration by parts to get
\begin{equation}
\pt\int\aLL \psi\,d_t x=\int\aLL( \Delta\psi+\pt\psi)\,d_t x,
\end{equation}
substituting which into $\eqref{tree}$ we have
\begin{equation}
\begin{split}
\left.\ptp\right|_{t_0}\int \mathcal{L}\psi\,d_t x
&\le\lim_{t\rightarrow t_0^+}\int\aLL(\Delta\psi+\pt\psi)\,d_tx=\left.\int\LL(\Delta\psi+\pt\psi)\,d_{t}x\right|_{t_0}
\end{split}
\end{equation}
%Integrating from $t_0$ to a later time $t$, we get
%\begin{equation}
%\left.(\int\aLL \psi\,d_t x)\right|^t_{t_0}\le\int_{t_0}^{t}\int\aLL(\Delta \psi+\pt\psi)\,d_t x\,dt,
%\end{equation}
%and hence
%$$\left.\ptt\right|_{t_0}\int\aLL \psi\,d_t x\le (\int(\aLL\Delta \psi +\left.\aLL\pt\psi)\,d_t x)\right|_{t_0}=(\int(\LL\Delta\psi +\left.\LL\pt \psi)\,d_t x)\right|_{t_0}$$
%This proves the $\eqref{banana}$. $\eqref{apple}$ follows by integrating $\eqref{banana}$. 
%from the fact that if F(t) is a continuous function defined on $[a,b]$, and $\ptt F(t)\le k(t)$ for all $t\in[a,b)$ and a continuous function $k(t)$, then there holds $F(b)\le F(a)+\int_{a}^{b}k(t)dt$.

%\begin{proof}
%Considering $F(t)-\int_{a}^{t}k(s)ds$ we can assume that $k(t)=0$. Suppose the conclusion doesn't hold, i.e. there exists $\varepsilon>0$ such that $F(b)=F(a)+\varepsilon$. Letting $G(t)=F(t)-(F(a)+\frac{\varepsilon}{b-a}(t-a))$, then $G(a)=G(b)=0$ and $\ptt G(t)\le-\frac{\varepsilon}{b-a}<0$ for all $t\in[a,b)$. In particular, we have $\left.\ptt\right|_{a} G(t)<0$ which then implies $G(t)$ has a negative minimum at some point $t_0$ in $(a,b)$. But we notice that $\left.\ptt\right|_{t_0}G(t)<0$, which implies there is some point near $t_0$, in its right hand side, takes an even smaller value than $t_0$. This is a contradiction and thus the claim is true.
%\end{proof}
\end{proof}

\begin{lem}\label{smooth of cut-off}
Let $\psi(x,t)$ be a non-negative continuous function which is compactly supported in $M$ for each $t$, and satisfies
$\Delta\psi\le u(x,t)$
and $\pt\psi\le v(x,t)$
in the barrier sense, 
where $v(x,t)$ is continuous with respect to $t$. 

Then for all $t$ we have
\begin{equation}\label{barrier}
\ptp\int \mathcal{L}(x,t)\psi(x,t)\,d_t x\le \int \mathcal{L}(x,t)(u(x,t)+v(x,t))\,d_t x
\end{equation}
\end{lem}

\begin{proof}
Let $t_0$ be an arbitrary time in $(a,b)$. Differentiating at $t_0$ by the product rule we get
\begin{equation}\label{two}
 \left.\ptp\right|_{t_0}\int \mathcal{L}(x,t)\psi(x,t)\,d_t x\le \int \mathcal{L}(x,t_0)v(x,t_0)\,d_{t_0} x+\left.\ptp\right|_{t_0}\int\LL(x,t)\psi(x,t_0)\,d_t x.
\end{equation}
Let $\aLL$ be the solution to the initial value problem
 \begin{equation}
(\ps-\Delta)\overline{\mathcal{L}}=\textnormal{scal}\overline{\mathcal{L}},\quad
\overline{\mathcal{L}}(\cdot,t_0)=\mathcal{L}(\cdot,t_0).
\end{equation}
Then $\aLL$ is smooth for all $t> t_0$. We calculate using the fact that barrier sense implies distributional sense:
\begin{equation}
\begin{split}
\left.\ptp\right|_{t_0}\int \mathcal{L}(x,t)\psi(x,t_0)\,d_t x&\le \left.\ptp\right|_{t_0}\int \aLL(x,t)\psi(x,t)\,d_t x\\
&\le\limsup_{t\rightarrow t_0^+}\pt\int\aLL(x,t)\psi(x,t_0)d_tx\\
&=\limsup_{t\rightarrow t_0^+}\int\Delta\aLL(x,t)\psi(x,t_0)d_tx\\
&\le\limsup_{t\rightarrow t_0^+}\int\aLL(x,t)u(x,t_0)d_tx\\
&=\int\LL(x,t_0)u(x,t_0)d_{t_0}x
\end{split}
\end{equation}
where we used the fact that barrier sense implies distributional sense
%It then suffices to show the second term in the RHS of $\eqref{two}$ is bounded from above by $\int \LL(x,t_0)u(x,t_0)\,d_{t_0}x$. To estimate it, we can assume $\LL$ is smooth by replacing $\LL$ with $\aLL$, which is defined as in last Lemma. By assumption, $\Delta\psi\le u(x,t)$ in the barrier sense. Recall the standard arguments, it also holds true in the sense of distribution (see e.g. \cite[Appendix]{appendix}). So integrating by parts as in the last lemma, we get 

%Apply the \cite[Proposition~2.5]{structure} to $\psi(x,t_0)$, for any $\varepsilon>0$, we can find $\psi_\varepsilon(x)$ such that $\psi_\varepsilon(x)\rightarrow \psi(x,t_0)$ uniformly and $\Delta \psi_\varepsilon(x)\le u(x,t_0)+\varepsilon$. Apply Lemma $\ref{g smooth}$ to each $\psi_{\varepsilon}$ and let $\varepsilon\rightarrow 0$, we have 
%\begin{equation*}
%(\left.\int \LL(x,s)\psi(x,t_0)\,d_s x)\right|^t_{t_0}=\lim_{\varepsilon\rightarrow 0}\left.(\int\LL(x,s)\psi_\varepsilon(x)\,d_s x)\right|^t_{t_0}
%\le\int^t_{t_{0}}(\int u(x,t_0)\LL(x,s)d_sx)ds
%\end{equation*}
%This implies immediately $\ptt|_{t_0} \int \LL(x,t)\psi(x,t_0)\le\int \LL(x,t_0)u(x,t_0)\,d_{t_0}x$. 
\end{proof}

\end{subsection}

\end{section}

%%%%%%%%%%%%%%%%%%%%%%%%%%%%%%%%%%%%%%%%%%%%%%%%%%%%%%%%%%%%%%%%%%%%%%%%

\begin{section}{Curvature Decay Lemma}
The main result in this section is Lemma $\ref{l: curvature decay}$, which provides a local estimate on the norm of the Riemann curvature tensor, under the assumption of a local bound for $\ell$. This lemma can be viewed as a weaker version of Theorem $\ref{t: theorem1}$ in the sense that we take the two conclusions of the existence of the Ricci flow and the bound of $\ell$, as additional hypotheses, and deduce the remaining conclusion about $|\Rm|$.

We need three ingredients in the proof of Lemma $\ref{l: curvature decay}$. One is the following Lemma, given in \cite[Lemma~5.1]{local3d} by a point-picking argument.

\begin{lem}\label{decay or no decay}
Given $c_0, r_0>0$, $n\in\mathbb{N}$, and take $\beta=\beta(n)>0$ as in Lemma $\ref{l: shrinking}$. Let $(M^n,g(t))$, $t\in[0,T]$ be a Ricci flow. Suppose for some $x_0\in M$ we have $B_{g(t)}(x_0,r_0)\subset\subset M$ for each $t\in[0,T]$. 

Then at least one of the following assertions is true:
\begin{enumerate}
  \item For each $t\in[0,T]$ with $t<\frac{r_0^2}{\beta^2 c_0}$, we have $B_{g(t)}(x_0,r_0-\beta\sqrt{c_0 t})\subset B_{g(0)}(x_0,r_0)$ and
     \begin{equation}
     |\textnormal{Rm}|_{g(t)}<\frac{c_0}{t}\quad\textnormal{on}\,\,\,B_{g(t)}(x_0,r_0-\beta\sqrt{c_0 t}).
     \end{equation}
   \item There exist $\bar{t}\in(0,T]$ with $\bar{t}<\frac{r_0^2}{\beta^2 c_0}$ and $\bar{x}\in B_{g(\bar{t})}(x_0,r_0-\frac{1}{2}\beta\sqrt{c_0 \overline{t}})$ such that 
      \begin{equation}
      Q:=|\textnormal{Rm}|_{g(\bar{t})}(\bar{x})\ge\frac{c_0}{\bar{t}},
      \end{equation}
    and
      \begin{equation}
      |\textnormal{Rm}|_{g(t)}(x)\le 4Q=4|\textnormal{Rm}|_{g(\bar{t})}(\bar{x}),
      \end{equation}
   whenever $d_{g(\bar{t})}(x,\bar{x})<\frac{\beta c_0}{8}Q^{-\frac{1}{2}}$ and $\bar{t}-\frac{1}{8}c_0 Q^{-1}\le t\le\bar{t}$.
\end{enumerate}
\end{lem}

The second ingredient we need is from \cite{local3d} which says the volume of a ball of fixed radius cannot decrease too rapidly under some curvature hypothesis. 

\begin{lem}\label{l: volume}
Given $K,\gamma,c_0,v_0,T>0$ and $n\in\mathbb{N}$, there exist positive constants $\varepsilon_0=\varepsilon_0(v_0,K,\gamma,n)$ and $\hat{T}=\hat{T}(v_0,c_0,K,\gamma,n)\ge 0$ such that the following holds: Let $(M^n,g(t)), t\in[0,T)$ be a Ricci flow such that $B_{g(t)}(x_0,\gamma)\subset\subset M$ for some $x_0\in M$ and all $t\in[0,T)$. Suppose $\Ric_{g(t)}\ge -K$ and $|\Rm|_{g(t)}\le\frac{c_0}{t}$ on $B_{g(t)}(x_0,\gamma)$ for all $t\in[0,T)$, and $Vol_{g(0)}B_{g(0)}(x_0,\gamma)\ge v_0$.

Then 
\begin{equation}
Vol_{g(t)}B_{g(t)}(x_0,\gamma)\ge\varepsilon_0
\end{equation}
for all $t\in[0,\hat{T}]\cap[0,T)$.
\end{lem}

The third ingredient is the following Lemma, which says that the asymptotic volume ratio of a weakly $\textnormal{PIC}_1$ ancient solution is zero. This is proved in \cite[Lemma~4.2]{almost}. We note that each  curvature condition listed in Theorem $\ref{t: theorem1}$ implies weakly $\textnormal{PIC}_1$, so the proof of Lemma $\ref{l: curvature decay}$ is uniform for all $\mathcal{C}$.

\begin{lem}\label{l: AVR}
Let $(M^n,g(t)), t\in(-\infty,0]$ be a nonflat ancient solution of the Ricci flow with bounded curvature satisfying weakly $\textnormal{PIC}_1$. Then it has non-negative complex sectional curvature. Furthermore, the volume growth is non-Euclidean, i.e. $\underset{r\rightarrow\infty}{\lim}r^{-n}Vol_{g(0)}B_{g(0)}(x,r)=0$ for all $x\in M$.
\end{lem}

We now states our main result of this section. In the proof we blow up a contradicting sequence to get a weakly $\textnormal{PIC}_1$ ancient solution with positive asymptotic volume ratio, which is impossible by Lemma $\ref{l: AVR}$.  

%%%%%%How about other curvatures? Do they also imply Ricci?

\begin{lem}\label{l: curvature decay}
(Curvature Decay Lemma). Given $v_0,K>0$, $0<\gamma<1$, and $n\in\mathbb{N}$, there exist positive constants $\tilde{T}=\tilde{T}(v_0,K,n,\gamma)$, $C_1=C_1(v_0,K,n,\gamma)$ and $\eta_0=\eta_0(v_0,K,n,\gamma)$ such that the following holds: Let $(M^n,g(t)), t\in[0,T]$ be a Ricci flow (not necessarily complete) such that $B_{g(t)}(x_0,1)\subset\subset M$ for each $t\in[0,T]$ and some $x_0\in M$, and
\begin{equation}
 Vol_{g(0)}B_{g(0)}(x_0,1)\ge v_0>0.
\end{equation}
Suppose further that
\begin{equation}
%\begin{dcases}
\ell(x,t)\le K\;\;\;\textnormal{on}\;\bigcup_{s\in[0,T]}B_{g(s)}(x_0,1),\;\;\;\textnormal{for all}\;\;t\in[0,T],
%\end{dcases}
\end{equation}
and $\gamma\in(0,1)$ is any constant. 

Then for all $t\in(0,T)\cap(0,\tilde{T})$, we have
\begin{equation}\label{decay}
|\textnormal{Rm}|_{g(t)}<\frac{C_1}{t}\;\;\;\textnormal{on}\;\;\;B_{g(t)}(x_0,\gamma),
\end{equation}
and
\begin{equation}
Vol_{g(t)}B_{g(t)}(x_0,1)\ge\eta_0\;\;\;\textnormal{and}\;\;\;\textnormal{inj}_{g(t)}(x_0)\ge\sqrt{\frac{t}{C_1}} 
\end{equation}
for all $t\in(0,\min(T,\tilde{T})]$.
\end{lem}

\begin{proof}
By Bishop-Gromov, $Vol_{g(0)}B_{g(0)}(x_0,\gamma)$ has a positive lower bound depending only on $v_0$, $K$ and $\gamma$. Applying Lemma $\ref{l: volume}$ to $g(t)$, we see that there exists $\eta_0>0$ depending only on $v_0$, $K$ and $\gamma$ such that for each $C_1<\infty$, there exist $\tilde{T}=\tilde{T}(v_0,\gamma,C_1)$ such that prior to time $\tilde{T}$ and while $|\textnormal{Rm}|_{g(t)}\le\frac{C_1}{t}$ still holds on $B_{g(t)}(x_0,\gamma)$, we have a lower volume bound
\begin{equation}
Vol_{g(t)}B_{g(t)}(x_0,1)\ge\eta_0.
\end{equation}
In particular, $\eta_0$ is independent of $C_1$. From this we deduce that is suffices to prove the lemma with the additional hypothesis that the equation above holds for each $t\in[0,T)$.

Let us assume that the lemma is false, even with the extra hypothesis. For some $v_0$, $K>0$ and $\gamma\in(0,1)$. Then for any sequence $c_{k}\rightarrow\infty$, we can find Ricci flows that fail the lemma with $C_1=c_k$ in an arbitrary short time, and in particular within a time $t_k$ that is sufficiently small so that $c_kt_k\rightarrow 0$ as $k\rightarrow\infty$. By reducing $t_k$ to the first time at which the desired conclusion fails, we have a sequence of Ricci flows $(M_k,\tilde{g}_k(t))$ for $t\in[0,t_k]$ with $t_k\rightarrow 0$, and even $c_kt_k\rightarrow 0$, and a sequence of points $x_k\in M_k$ with $B_{\tilde{g}_k(t)}(x_k,1)\subset\subset M_k$ for each $t\in[0,t_k]$, such that
\begin{equation}
\begin{split}
Vol_{\tilde{g}_k(t)}B_{\tilde{g}_k(t)}(x_k,1)&\ge\eta_0,\,\,\,\textnormal{for all}\,\,\, t\in [0,t_k],
\end{split}
\end{equation}
\begin{equation}
\ell(x,t)\le K,\,\,\,\textnormal{on}\,\,\,\bigcup_{s\in[0,t_k]}B_{\tilde{g}_k(s)}(x_k,1)\,\,\,\textnormal{for all}\,\,\,t\in[0,t_k],
\end{equation}
and
\begin{equation}
|\textnormal{Rm}|_{\tilde{g}_k(t)}<\frac{c_k}{t}\,\,\,\textnormal{on}\,\,\, B_{\tilde{g}_k(t)}(x_k,\gamma)\,\,\, \textnormal{for all}\,\,\, t\in[0,t_k],
\end{equation}
but so that 
\begin{equation}\label{hi}
|\textnormal{Rm}|_{\tilde{g}_k(t_k)}=\frac{c_k}{t_k}\,\,\,\textnormal{at some point in}\,\,\,\overline{B_{\tilde{g}_k(t)}(x_k,\gamma)}.
\end{equation}
For sufficiently large $n$, we have $\beta\sqrt{c_kt_k}<\frac{1-\gamma}{2}$. We apply Lemma $\ref{decay or no decay}$, to each $\tilde{g}_k(t)$ with $r_0=\frac{1+\gamma}{2}$ and $c_0=c_k$, then it follows by $\eqref{hi}$ that Assertion $1$ there cannot hold, and thus Assertion $2$ must hold for each $n$, giving time $\bar{t}_k\in(0,t_k]$ and points $\bar{x}_k\in B_{\tilde{g}_k(\bar{t}_k)}(x_k, r_0-\frac{1}{2}\beta\sqrt{c_k\overline{t}_k})$ such that
\begin{equation}\label{blessed}
|\textnormal{Rm}|_{\tilde{g}_k(t)}(x)\le 4|\textnormal{Rm}|_{\tilde{g}_k(\bar{t}_k)}(\bar{x}_k)
\end{equation}
on $B_{\tilde{g}(\bar{t}_k)}(\bar{x}_k,\frac{\beta c_k}{8}Q_{k}^{-\frac{1}{2}})$, for all $t\in[\bar{t}_k-\frac{1}{8}c_kQ_k^{-1}, \bar{t}_k]$, where $Q_k:=|Rm|_{\tilde{g}_k(\bar{t}_k)}(\bar{x}_k)\ge \frac{c_k}{\bar{t}_k}\rightarrow\infty$. We also notice that $B_{\tilde{g}(\bar{t}_k)}(\bar{x}_k,\frac{\beta c_k}{8}Q_{k}^{-\frac{1}{2}})\subset B_{\tilde{g}(\bar{t}_k)}(x_k,1)$, thus
\begin{equation}\label{happy}
 \ell(x,t)\le K
\end{equation}
on $B_{\tilde{g}(\bar{t}_k)}(\bar{x}_k,\frac{\beta c_k}{8}Q_{k}^{-\frac{1}{2}})\times[\bar{t}_k-\frac{1}{8}c_kQ_k^{-1}, \bar{t}_k]$.
The above conditions at $\bar{t}_k$, together with Bishop-Gromov, imply that we have uniform volume ratio control
\begin{equation}\label{e: ratio}
\frac{Vol_{\tilde{g}_k(\bar{t}_k)}B_{\tilde{g}_k(\bar{t}_k)}(\bar{x}_k,r)}{r^n}\ge\eta>0
\end{equation}
for all $0<r<\frac{1-\gamma}{2}$, where $\eta$ depends on $\eta_0$, $K$ and $\gamma$.
A parabolic rescaling on $B_{\tilde{g}(\bar{t}_k)}(\bar{x}_k,\frac{\beta c_k}{8}Q_{k}^{-\frac{1}{2}})\times [\bar{t}_k-\frac{1}{8}c_kQ_k^{-1}, \bar{t}_k]$ gives new Ricci flows defined by 
$$g_k(t):=Q_k\tilde{g}_k(\frac{t}{Q_k}+\bar{t}_k)$$ for $t\in[-\frac{1}{8}c_k,0]$. The scaling factor is chosen so that $|\textnormal{Rm}|_{g_k(0)}(\bar{x}_k)=1$. By $\eqref{blessed}$, the curvature of $g_k(t)$ is uniformly bounded on $B_{g_{k}(0)}(\bar{x}_k,\frac{1}{8}\beta c_k)\times [-\frac{1}{8}c_k,0]$. Condition $\eqref{happy}$ transforms to
\begin{equation}\label{PIC}
\ell(x,t)\le \frac{K}{Q_k}\rightarrow 0
\end{equation}
on $B_{g_{k}(0)}(\bar{x}_k,\frac{1}{8}\beta c_k)\times[-\frac{1}{8}c_k,0]$.
The volume ratio $\eqref{e: ratio}$ gives 
\begin{equation}
\frac{Vol_{g_k(0)}B_{g_k(0)}(\bar{x}_k,r)}{r^n}\ge\eta>0
\end{equation}
for all $0<r<\frac{1-\gamma}{2}Q_k^{\frac{1}{2}}\rightarrow\infty$.

With this control we can apply Hamilton's compactness theorem to give convergence $(M_k,g_k(t),\bar{x}_k)\rightarrow(N,g(t),x_{\infty})$, for some complete bounded-curvature Ricci flow $(N,g(t))$, for $t\in(-\infty,0]$, and $x_{\infty}\in N$. 

Moreover, the last volume equation passes to limit to force $g(t)$ to have positive asymptotic volume ratio. From $\eqref{PIC}$ we know that $g(t)$ is a nonflat ancient solution of Ricci flow with bounded curvature satisfying weakly $\textnormal{PIC}_1$. This contradicts Lemma $\ref{l: AVR}$ that the volume ratio of $(N,g(t))$ vanishes, and thus shows the first part of the Lemma. For the second part, we choose $\gamma=\frac{1}{2}$, then $Vol_{g(t)}B_{g(t)}(x_0,\frac{1}{2})\ge\eta_0>0$. The injectivity radius estimate of  Cheeger-Gromov-Taylor \cite{inj} and the Bishop-Gromov comparison then tell us $\textnormal{inj}_{g(t)}(x)\ge i_0\sqrt{t}$ for some $i_0=i_0(\eta_0,C)>0$.
\end{proof}

%\begin{lem}{(cf. \cite[Lemma~4.1]{mollification})}\label{l: inj}
%Let $(N^n,g(t))_{t\in[0,T]}$ be a smooth Ricci flow such that for some fixed $x\in N$ we have $B_{g(t)}(x,1)\subset\subset N$ for all $t\in[0,T]$, and so that 
 %\begin{enumerate}
 %\item $Vol_{g(0)}B_{g(0)}(x,1)\ge v_0>0$, and
 %\item $\ell(x,t)\le 1$ on $B_{g(t)}(x,1)$ for all $t\in[0,T]$.
 %\end{enumerate}
%Then there exist $C_0=C_0(v_0)$ and $\hat{T}=\hat{T}(v_0)>0$ such that $|\textnormal{Rm}|_{g(t)}(x)\le\frac{C_0}{t}$, and $\textnormal{inj}_{g(t)}(x)\ge\sqrt{\frac{t}{C_0}}$ for all $0<t\le \min(\hat{T},T)$.
%\end{lem}
%\begin{proof}
%Using Lemma $\ref{l: curvature decay}$ and Bishop-Gromov comparison we get $C_0=C_0(v_0)\ge 1$, $\hat{T}=\hat{t}(v_0)>0$ and $\eta_0=\eta_0(v_0)$ such that $|\textnormal{Rm}|_{g(t)}\le\frac{C_0}{t}$ on $B_{g(t)}(x,\frac{1}{2})$ and $Vol_{g(t)}B_{g(t)}(x,\frac{1}{2})\ge\eta_0$ for all $t\in(0,\min(\hat{T},T)]$. The injectivity radius estimate of Cheeger-Gromov-Taylor \cite{inj} and the Bishop-Gromov comparison then tell us (after scaling $g(t)$ to $\tilde{g}=\frac{1}{t}g(t)$ for each $t$, and then scaling back), that there exists $i_0=i_0(\eta_0,C_0)=i_0(v_0)>0$ such that $\textnormal{inj}_{g(t)}(x)\ge i_0\sqrt{t}$.
%\end{proof}

\end{section}

%%%%%%%%%%%%%%%%%%%%%%%%%%%%%%%%%%%%%%%%%%%%%%%%%%%%%%%%%%%%%%%%%%%%%%%%
\begin{section}{A cut-off function}
In this section we construct a cut-off function on manifolds (not assumed to be complete) evolving by Ricci flow, which helps to localize the integration estimates in section 7.

\begin{lem}\label{l: cutoff}
Given $n\in\mathbb{N}$, $c_0, K>0$, $0<T<1$, $0<R<1$, $0<r<\frac{1}{10}$ with $\beta\sqrt{c_0 T}\le\frac{1}{4}r$, where $\beta=\beta(n)$ is from the Shrinking Lemma, there exists positive constant $C=C(n,K,v_0)$ such that the following holds: 
Let $(M^n,g(t)), t\in[0,T]$ be a smooth Ricci flow such that $B_{g(0)}(x_0,R+r)\subset\subset M$, and on $B_{g(0)}(x_0,R+r)\times[0,T]$,
\begin{equation}
\textnormal{Ric}_{g(t)}(x)\ge -K\quad\textnormal{and}\quad
|\textnormal{Rm}|_{g(t)}\le\frac{c_0}{t},
\end{equation}
and for all $\delta\in[0,r]$ and $x\in B_{g(0)}(x_0,R)$ we have
\begin{equation}\label{add}
Vol_{g(0)}B_{g(0)}(x,\delta)\ge v_0 \delta^n. 
\end{equation}

Then there exists a continuous function $\phi(y,s): M\times[0,T]\longrightarrow \mathbb{R}$ with the following properties:
\begin{description}
 \item[(P1)] $supp\,\phi(\cdot,s)\subset B_{g(0)}(x_0,R)$ for all $s\in[0,T]$.
 \item[(P2)] $\nabla\phi$ exists a.e. and $|\nabla\phi|\le Cr^{-(n+1)}$.
 \item[(P3)] $\Delta\phi\le Cr^{-(2n+2)}$ in the barrier sense. 
 \item[(P4)] $\psp\phi\le Cr^{-n}$.
 \end{description}
%for each $s\in[0,T]$, $supp\,\phi(\cdot,s)\subset B_{g(0)}(x_0,R)$ and $\nabla\phi$ exists almost everywhere, satisfying
%\begin{equation}\label{foothill}
%|\nabla\phi|\le\frac{C(K,v_0,n)}{(r)^{n+1}}
%\end{equation}
%Moreover, we have
%\begin{equation}\label{crossroad}
%\Delta\phi\le\frac{C(K,v_0,n)}{(r)^{2n+2}},\quad\textnormal{and}\,\,\,
%\ps\phi\le\frac{C(K,v_0,n)}{(r)^{n}}
%\end{equation}
%hold in the barrier sense and the right barrier sense, respectively. 

Moreover, We have the inclusions:
\begin{equation}\label{e: last}
B_{g(s)}(x_0,R-\tfrac{5}{4}r)\subset B_{g(0)}(x_0,R-r)\subset\{y\in M \,|\, \phi(y,s)=1\} 
\end{equation}
for all $s\in[0,T]$.
\end{lem}

\begin{proof}
Let $f:\mathbb{R}\longrightarrow\mathbb{R}$ be a non-increasing smooth function such that $f(z)=1$ for all $z<\frac{1}{4}$ and $f(z)=0$ for all $z>\frac{1}{2}$. Let $F:\mathbb{R}\longrightarrow\mathbb{R}$ be a non-decreasing and convex smooth function such that $F(z)=0$ for all $z\le 0$ and $F(1)=1$. 
Let $C_0$ be a constant such that $|f'|,|f''|,|F'|,|F''|\le C_0$. Hereafter we use the same letter $C$ to denote the constants depending on $K$, $v_0$, $n$. 

Let $\{p_k\}_{k=1}^{N}$ be a maximal $\frac{r}{4e^K}$-separated set in the annulus $A:=B_{g(0)}(x_0,R)-B_{g(0)}(x_0,R-\frac{1}{4}r)$ with respect to $g(0)$.
By a $\varepsilon$-separated set we mean a set in which the points are at least $\varepsilon$-distant from each other. 
%By the volume assumption we see that the volume of $\frac{r}{4e^K}$-ball centered at each $p_k$ is bounded below by $v_0(\frac{r}{4e^K})^n$ by volume comparison. 
It's clear that the $\varepsilon/2$-balls of points in a $\varepsilon$-separated set are disjoint pairwise.  
%we have $\sum_{k=1}^N Vol_{g(0)}B_{g(0)}(p_k,\frac{r}{8e^K})\le Vol_{g(0)}B_{g(0)}(x_0,R+r)$. 
By volume comparison we see that $Vol_{g(0)}B_{g(0)}(x_0,R)\le C$, 
and furthermore by $\eqref{add}$ $Vol_{g(0)}B_{g(0)}(p_k,\frac{r}{4e^K})\ge Cr^n$. Hence we have $N\le Cr^{-n}$.

\begin{claim}\label{l: covering}
$A\subset\bigcup\limits_{k=1}^{N}B_{g(s)}(p_k,\frac{r}{4})$ for all $s\in[0,T]$.
\end{claim}
\begin{proof}[Proof of Claim \ref{l: covering}]
By the choice of $\{p_k\}_{k=1}^N$ we see that $A\subset\bigcup\limits_{k=1}^{N}B_{g(0)}(p_{k},\frac{r}{4e^{K}})$. For each $p_k$, the triangle inequality implies that $B_{g(0)}(p_k,\frac{r}{2})\subset\subset B_{g(0)}(x_0,R+r)$ where $|\textnormal{Rm}|_{g(s)}\le\frac{c_0}{s}$ and $\textnormal{Ric}_{g(s)}\ge-K$ holds true for all $s\in[0,T]$. 
Applying the Shrinking Lemma to $g(t)$, we find that $B_{g(s)}(p_k,\frac{r}{2}-\beta\sqrt{c_0 s})\subset B_{g(0)}(p_k,\frac{r}{2})$ for all $s\in[0,T]$ and in particular $B_{g(s)}(p_k,\frac{r}{4})\subset B_{g(0)}(p_k,\frac{r}{2})$ due to $\beta\sqrt{c_0 T}\le\frac{1}{4}r$. So $\Ric\ge-K$ holds on $B_{g(s)}(p_k,\frac{r}{4})$, which gives the condition we need in order to apply the Expanding Lemma to the Ricci flow on $B_{g(s)}(p_k,\frac{r}{4})\times [0,s]$, giving $B_{g(s)}(p_k,\frac{r}{4})\supset B_{g(0)}(p_k,\frac{r}{4e^K})$, and thus proves the claim.
\end{proof}

By the Shrinking Lemma and triangle inequality, we have 
$B_{g(s)}(p_k,\frac{r}{2})\subset B_{g(0)}(p_k,r)\subset B_{g(0)}(x_0,R+r)$. %For $y\notin B_{g(s)}(p_k,\frac{r}{2})$, we have $f(\frac{d_{g(s)}(p_k,y)}{r})=0$. 
In view of this together with the definition of $f$, we define the following continuous function on $M$:
\begin{equation}
f_k(y,s)=
 \left\{\begin{matrix}
 f\left(\frac{d_{g(s)}(p_k,y)}{r}\right)\quad  & \textnormal{for}\quad y\in B_{g(0)}(p_k,r)\;;\\
 0\quad & \textnormal{for}\quad y\notin B_{g(0)}(p_k,r).
 \end{matrix}\right.
\end{equation}

By Claim $\ref{l: covering}$, for each point $y\in A$ and $s\in[0,T]$, there is some $k$ such that $y\in B_{g(s)}(p_k,\frac{r}{4})$, $f_k(y,s)=1$ and $F(1-\sum_{k=1}^{N}f_k(y,s))=0$. Based on this we define the following continuous function on $M$:
\begin{equation}
\phi(y,s)=
 \left\{\begin{matrix}
F(1-\sum\limits_{k=1}^{N}f_k(y,s))\quad & \textnormal{for}\quad y\in B_{g(0)}(x_{0},R)\;;\\
0\quad & \textnormal{for}\quad y\notin B_{g(0)}(x_{0},R)\;.
 \end{matrix}\right.
\end{equation}

%\begin{remark}\label{r: distance}
%If we have a Riemannian manifold $(N,g)$, and we are considering a ball $B_{g}(x_0,\hat{R})\subset\subset N$, then we can make sense of the distance between two points $x, y\in B_{g}(x_0,\hat{R})$ in two ways-either as the infimum length of connecting paths that remain within $B_{g}(x_0,\hat{R})$, or as the infimum length of such paths that may stray anywhere within $N$. In general, the letter distance will be shorter. However, we observe that if we are sure that $x$ and $y$ lie within the half ball $B_{g}(x_0,\frac{\hat{R}}{2})$, then these two distances must agree, and there exists a minimizing geodesic between $x$ and $y$ that remains within $B_g(x_0,R)$. 

%Here in the definition of $\phi$, we consider $d_{g(s)}(p_k,y)$ as the infimum of length of all paths connecting $p_k$ and $y$ which remain within $B_{g(0)}(x_0,R+r)$ for each time $s$, so $d_{g(s)}(p_k,\cdot)$ is a continuous function on $B_{g(0)}(x_0,R+r)\times[0,\min(\tilde{T},T)]$. In particular, for any $y\in B_{g(s)}(p_k,\frac{1}{2}r)\subset B_{g(s)}(p_k,r)\subset\subset B_{g(0)}(x_0,R+2r)$, as explained above, we can find a minimizing geodesic connecting $p_k$ and $y$ that remains with $B_{g(s)}(p_k,r)\subset B_{g(0)}(x_0,R+2r)$ where we have $\textnormal{Ric}\ge-K$.
%\end{remark}

It's clear that $\phi(y,s)$ satisfies (P1). 
Below we abbreviate $d_{g(s)}(p_k,y)$ by $d_k$,  $f'(\frac{d_{g(s)}(p_k,y)}{r})$ by $f'_k$, and $f''(\frac{d_{g(s)}(p_k,y)}{r})$ by $f''_k$.
Using that
\begin{equation}
\nabla\phi=-F'\cdot\sum_{k=1}^N f'_k\cdot r^{-1}\cdot\nabla d_k,
\end{equation}
and taking into account that $\nabla d_k$ exists a.e. with $|\nabla d_k|=1$, and $N\le C\cdot r^{-n}$, we see that $\nabla\phi$ exists a.e. and
\begin{equation}
|\nabla\phi|\le C\cdot r^{-(n+1)}.
\end{equation}

%\begin{equation}\label{derivatives}
%\begin{split}
%|\nabla\phi|&=F'\cdot\frac{1}{r}\sum_{k=1}^{N}|f'(\frac{d_k}{r})|\cdot|\nabla d_k|\le\frac{CN}{r}\\
%\Delta \phi&=F''\cdot|\nabla \sum_{k=1}^{N}f_k|^2-F'\cdot\Delta(\sum_{k=1}^{N} f_k)\le \frac{CN^{2}}{r^{2}}+\frac{CN}{r^{2}}-F'\cdot\frac{1}{r}\sum_{k=1}^{N}f'(\frac{d_k}{r})\Delta d_{k}\\
%\ps\phi&=-F'\cdot\frac{1}{r}\sum_{k=1}^{N}f'_k\cdot\ps d_k
%\end{split}
%\end{equation}
To estimate $\ps\phi(y,s)$ and $\Delta \phi(y,s)$, we may assume $y \in B_{g(s)}(p_k,\frac{1}{2}r)-B_{g(s)}(p_k,\frac{1}{4}r)$ without loss of generality. Because otherwise  $f'(\frac{d_k(y,s)}{r})=0$, and hence $\ps\phi(y,s)=\Delta \phi(y,s)=0$. By the Shrinking Lemma and the choice of $p_k$ we have 
\begin{equation}
B_{g(s)}(p_k,\frac{1}{2}r)\subset B_{g(0)}(p_k,r)\subset B_{g(0)}(x_0,R+r).
\end{equation}
So the minimizing geodesic connecting $y$ and $p_k$ with respect to $g(s)$ remains within $B_{g(0)}(x_0,R+r)$ where $\textnormal{Ric}_{g(s)}\ge -K$. Hence by the Laplacian comparison and noting that $d_{g(s)}(y,p_k)\ge\frac{1}{4}r$, we have
\begin{equation}
\Delta d_{g(s)}(p_{k},y)\le (n-1)\sqrt{K}\textnormal{coth}(\sqrt{K}d_{g(s)}(p_{k},y))\le\frac{C}{r}
\end{equation}
in the barrier sense. Then using that
\begin{equation}
\Delta\phi=F''|\sum_{k=1}^N f'_k\cdot r^{-1}\cdot\nabla d_k|^2-F'\cdot\sum_{k=1}^N(f''_k\cdot r^{-2}\cdot|\Delta d_k|^2+f'_k\cdot r^{-1}\cdot\Delta d_k),
\end{equation}
and noting $f'\le 0$, $F'\ge 0$, we can estimate
\begin{equation}
\Delta\phi\le C\cdot r^{-(2n+2)}.
\end{equation}
We see from the Ricci flow equation that
\begin{equation}
\psp d_{g(s)}(p_k,y)\le K d_{g(s)}(p_k,y)\le \frac{1}{2}Kr,
\end{equation}
and using that 
\begin{equation}
\psp\phi=-F'\sum_{k=1}^N f'_k\cdot r^{-1}\cdot\psp d_k,
\end{equation}
we obtain
\begin{equation}
\psp\phi\le C\cdot r^{-n}.
\end{equation}

It remains to prove the inclusion $\eqref{e: last}$. The first inclusion is a consequence of the Shrinking Lemma and $\beta\sqrt{c_0 T}\le\frac{1}{4}r$. To prove the second inclusion, we note by triangle inequality that 
\begin{equation}
B_{g(0)}(x_0,R-r)\cap \bigcup_{k=1}^N B_{g(0)}(p_k,\frac{3}{4}r)=\emptyset,
\end{equation}
and by the Shrinking Lemma, 
\begin{equation}
B_{g(s)}(p_k,\frac{1}{2}r)\subset B_{g(0)}(p_k,\frac{1}{2}r+\beta\sqrt{c_0 T})\subset B_{g(0)}(p_k,\frac{3}{4}r)
\end{equation}
for each $k$ and $s\in[0,T]$. Thus for all $s\in[0,T]$,
\begin{equation}\label{ten}
B_{g(0)}(x_0,R-r)\cap \bigcup_{k=1}^N B_{g(s)}(p_k,\frac{1}{2}r)=\emptyset.
\end{equation}
Then the second inclusion in $\eqref{e: last}$ follows immediately from $\eqref{ten}$ and the definitions of $f$ and $\phi$.
%For each $k$, we find that $f_k(\cdot,s)=0$ on $B_{g(0)}(x_0,R- r)$ for all $s\in[0,T]$. This yields the second inclusion in $\eqref{e: last}$.

\end{proof}
%%%%%%%we are very careful because it's only safe to talk about distance with respect to g(0)

\end{section}

%%a new section: heat kernel estimate

\begin{section}{Heat kernel estimates for Ricci flow in expansion}

\begin{subsection}{An upper bound for the heat kernel of Ricci flow}
Let $(M,g(t))$, $t\in[0,T]$, be a complete Ricci flow. Hereafter we denote by $G(x,t;y,s)$, with $x,y\in M$, $0\le s<t\le T$, the heat kernel corresponding to the backwards heat equation coupled with the Ricci flow. This means that for any fixed $(x,t)\in M\times [0,T]$ we have
\begin{equation}
(\ps + \Delta_{y,s})G(x,t;y,s)=0\,\,\,\textnormal{and}\,\,\, \lim\displaylimits_{s\nearrow t}G(x,t;y,s)=\delta_x(y)
\end{equation}
Then for any fixed $(y,s)\in M\times[0,T]$ we can compute that $G(\cdot,\cdot;y,s)$ is the heat kernel associated to the conjugate equation
\begin{equation}
(\frac{\partial}{\partial t}-\Delta_{x,t}-\scal_{g(t)})G(x,t;y,s)=0 \,\,\,\textnormal{and}\,\,\,\lim\displaylimits_{t\searrow s}G(x,t;y,s)=\delta_y(x).
\end{equation}
Note that in literatures it is more common to consider the fundamental solution of the conjugate heat equation $\frac{\partial}{\partial t}u+\Delta_{x,t}u-\scal u=0$. 
%Hereafter $d_t$ and $d_t z$ will denote the Riemannian distance and the volume element, respectively, for the metric $g(t)$.
$G(x,t;y,s)$ has the following property
\begin{equation}\label{e: crucial}
\int\displaylimits_{M}G(x,t;y,s)\,d_t x=1 \,\,\,\textnormal{for all}\,\,\, 0\le s< t\le T.
\end{equation}
In the compact case, this follows from the following simple calculation:
\begin{equation}
 \pt\int_M G(x,t;y,s)\,d_tx=\int_M((\Delta_{x,t}+\scal_{g(t)})G(x,t;y,s)-G(x,t;y,s)\,\scal_{g(t)})\,d_t x=0.
\end{equation}
The general case follows using an exhaustion and limiting argument.

The heat kernel $G$ has a Gaussian bound by the following proposition from \cite{almost}. 

\begin{prop}\label{prop: heat kernel}Given $n\in\mathbb{N}$ and $A>0$, there is a constant $C=C(n,A)<\infty$ such that the following holds: Let $(M^n,g(t))$, $t\in[0,T]$, be a complete Ricci flow satisfying
\begin{equation}\label{invariant}
|\textnormal{Rm}|_{g(t)}\le\frac{A}{t}\;\;\;\textnormal{and}\;\;\; Vol_{g(t)}B_{g(t)}(x,\sqrt{t})\ge\frac{(\sqrt{t})^n}{A}
\end{equation}
for all $(x,t)\in M\times(0,T]$. Then 
\begin{equation}\label{lotus}
G(x,t;y,s)\le\frac{C}{(t-s)^{\frac{n}{2}}}exp(-\frac{d_s^2(x,y)}{C(t-s)})\,\,\,\textnormal{for all}\,\,\, 0\le s< t\le T.
\end{equation}
\end{prop}

\begin{remark}
We note that $\eqref{invariant}$ is invariant under rescaling and time shifting in the sense that for the Ricci flow $\hat{g}(\tau)=\frac{1}{t-s}g(\tau(t-s)+s), \tau\in[0,1]$, where $0\le s<t\le T$, the condition $\eqref{invariant}$ still holds true. The right-hand side of the second bound in $\eqref{invariant}$ may change by a controlled factor due to a volume comparison argument.     
\end{remark}

%\begin{lem}\label{l: rescaling}
%Let $(M^n,g(t)),t\in[0,T]$ be a Ricci flow satisfying
%\begin{equation}\label{joe}
%|\textnormal{Rm}|_{g(t)}\le\frac{A}{t}\,\,\,\textnormal{and}\,\,\, Vol_{g(t)}B_{g(t)}(x,\sqrt{t})\ge\frac{(\sqrt{t})^n}{A}.
%\end{equation}
%Then for $0\le s<t\le T$, and the rescaled Ricci flow $\hat{g}(\tau)=\frac{1}{t-s}g(\tau(t-s)+s), \tau\in[0,1]$, we have
%\begin{equation}
%|\textnormal{Rm}|_{\hat{g}(\tau)}\le\frac{\hat{A}}{\tau}\,\,\,\textnormal{and}\,\,\, Vol_{\hat{g}(\tau)}B_{\hat{g}(\tau)}(x,\sqrt{\tau})\ge\frac{(\sqrt{\tau})^n}{\hat{A}}
%\end{equation}
%for a constant $\hat{A}=\hat{A}(A,n)$.
%\end{lem}

%\begin{proof}
%Let $T=\tau(t-s)+s$, then $\hat{g}(\tau)=\frac{1}{t-s}g(T)$, which leads to
%\begin{equation}
%\begin{split}
%|\textnormal{Rm}|_{\hat{g}(\tau)}&=(t-s)|\textnormal{Rm}|_{g(T)}\le \frac{A(t-s)}{\tau(t-s)+s}<\frac{A}{\tau}\\
%Vol_{\hat{g}(\tau)}B_{\hat{g}(\tau)}(x,\sqrt{\tau})&=\frac{1}{(t-s)^{\frac{n}{2}}}Vol_{g(T)}B_{g(T)}(x,\sqrt{\tau(t-s)})
%\end{split}
%\end{equation}
%Combining the condition $\eqref{joe}$ at time $T$ with Bishop comparison, we find a constant $\hat{A}=\hat{A}(A,n)$ such that
%$Vol_{g(T)}B_{g(T)}(x,\sqrt{\tau(t-s)})\ge\frac{(\tau(t-s))^{\frac{n}{2}}}{\hat{A}}$. Substitute this into last equation, we get
%$Vol_{\hat{g}(\tau)}B_{\hat{g}(\tau)}(x,\sqrt{\tau})\ge\frac{\tau^{\frac{n}{2}}}{\hat{A}}$.

%\end{proof}

\end{subsection}

\begin{subsection}{Generalized heat kernel of Ricci flow in extension and its upper bound}
\begin{defn}(Ricci flow in expansion)
We say $(\{M_j\}_{j=1}^m,\{g_j(t)\}_{j=1}^{m},\nu)$ is a Ricci flow in expansion, if for each $j$, $(M_j,g_j(t))$ is a complete Ricci flow defined on $[t_j,t_{j+1}]$ with $t_0=0$, $t_{j+1}=\nu\,t_j$ and $M_0\supset M_1\supset M_2\supset...\supset M_{m}$. Moreover, at each $t_{j+1}$ we have $g_{j+1}(t_{j+1})\ge g_j(t_{j+1})$ everywhere on $M_{j+1}$. 
\end{defn}

We call each $t_j$ a expanding time. In the following discussion we will often need to distinguish metrics $g_{j-1}(t_j)$ and $g_j(t_j)$. Without ambiguity, we use $t_j^+$ whenever referring to any geometric quantity with respect to $g_j(t_j)$, and $t_j^-$ for $g_{j-1}(t_j)$ respectively. For example, $B_{t_j^+}(x,r)$ denotes a $r$-ball centered at $x$ with respect to $g_j(t_j)$ and $M_{t_j^+}$ denotes $M_j$.

%\begin{remark}
%For a complete Ricci flow $(M^n,g(t)),t\in[0,t_1]$, choose a subset $U$ as in Lemma $\ref{l: conformal}$ and apply the lemma to $U$, we get a complete metric with bounded curvature on a subset $\tilde{M}\subset U$. On $\tilde{M}$ we can run a short time Ricci flow $\hat{g}(t)$ for $t\in[t_1,t_2]$. Because the conformal change enlarges the distance, we obtain a Ricci flow in extension.
%\end{remark}

%For a complete Ricci flow $(M,g(t))$, the corresponding heat kernel $G(x,t;y,s)$ satisfies the following reproduction formula,
%\begin{equation}
% G(x,t;y,s)=\int_{M}G(x,t;z,u)G(z,u;y,s)d_u z
%\end{equation}
%for any $s<u<t$. Inspired by this, we carefully define a heat kernel for the above Ricci flow in expansion. Then this heat kernel agrees with the ordinary heat kernel when the Ricci flow in expansion is actually a complete Ricci flow. More specifically, we have the following definition.

\begin{defn}\label{gene}
(Generalized heat kernel) Let $(\{M_j\}_{j=1}^{n},\{g_j(t)\}_{j=0}^{n},\nu)$ be a Ricci flow in expansion. For any $x\in M_i$ and $t\in(t_i,t_{i+1}]$, we define the generalized heat kernel $G(x,t;\cdot,\cdot)$ as follow: First, $G(x,t;y,s)$ is the standard heat kernel for all $y\in M_i$ and $s\in [t_i,t)$. Next, suppose $G(x,t;z,s')$ has been defined for all $z\in M_{j}$ and $s'\in [t_{j},t_{j+1})$ for some $j\le i-1$. Then for $y\in M_{j-1}$ and $s\in [t_{j-1},t_{j})$, we set
\begin{equation}
G(x,t;y,s)=\int_{M_{t_j^+}}G(x,t;z,t_{j})G(z,t_{j};y,s)d_{t_j^-}z.
\end{equation}
Inductively, $G(x,t;\cdot,\cdot)$ is defined on $(\bigcup_{j=0}^{i-1}M_j\times[t_j,t_{j+1}))\cup M_i\times[t_i,t)$  (see Figure 1). It's easy to see that $G(x,t;\cdot,\cdot)$ is continuous on all over its domain, and smooth on each $M_j\times(t_j,t_{j+1})$ for $j\le i-1$ and on $M_i\times(t_i,t)$.
\end{defn}

The goal of this section is to derive a Gaussian bound for the generalized heat kernel. A crucial fact in the proof is the $L^1$-norm of $G(\cdot,t_j;y,t_{j-1})$ is not bigger than 1 for all $t$, that is,
\begin{equation}\label{e: crucial2}
\int_{M_{t_j^+}}G(x,t_{j};y,t_{j-1})\,d_{t_j^-}x\le\int_{M_{t_j^-}}G(x,t_{j};y,t_{j-1})\,d_{t_j^-}x=1
\end{equation}
for any $y\in M_{j-1}$.

%%Note that for the ordinary heat kernel $G(x,t;y,s)$ of the complete Ricci flow piece $(M_i,g_i(t))_{t\in[t_i,t_{i+1}]}$, as an application of Proposition above, we have that
%%\begin{equation}
%%G(x,t;y,s)\le\frac{C}{t^{\frac{n}{2}}}exp(-\frac{d_s^2(x,y)}{Ct}) ~~for %%~all~ x,y\in M_i ~and~ t_i\le s<t\le t_{i+1}
%%\end{equation}

\begin{figure}[t]
\centering
\includegraphics[width=1.0\linewidth]{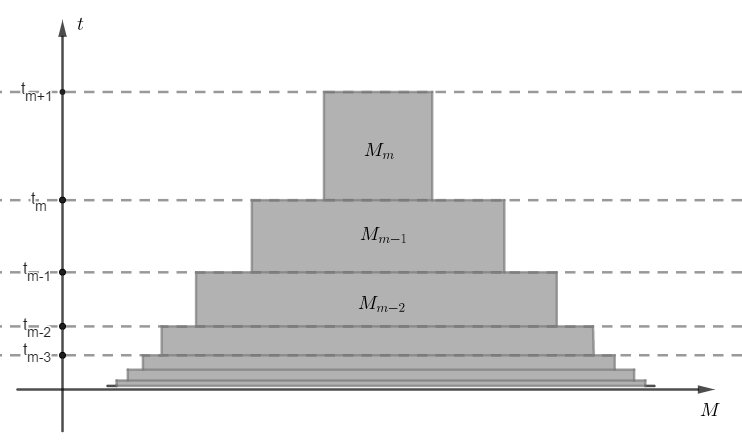}
\caption{Ricci flow in expansion}
\label{f:RF}
\end{figure}

\begin{prop}\label{generalized heat kernel}
Given $n\in\mathbb{N}$, $A>0$, and $\nu>1$, there is a constant $C=C(n,A,\nu)<\infty$ such that the following holds: Let $(\{M_j\}_{j=0}^{m},\{g_j(t)\}_{j=0}^{m},\nu)$ be a Ricci flow in expansion such that for each $j$ we have
\begin{equation}\label{trader}
|\textnormal{Rm}|_{g_j(t)}\le\frac{A}{t}\,\,\,\textnormal{and}\,\,\, Vol_{g_j(t)}B_{g_j(t)}(x,\sqrt{t})\ge\frac{t^\frac{n}{2}}{A}
\end{equation}
for all $x\in M_j$ and $t\in[t_j,t_{j+1}]$. Then for any pairs $(x,t)$ and $(y,s)$ such that $G(x,t;y,s)$ is well defined as above, we have
\begin{equation}\label{e: heat kernel}
G(x,t;y,s)\le\frac{C}{(t-s)^{\frac{n}{2}}}exp\left(-\frac{d_{s^+}^2(x,y)}{C(t-s)}\right).
\end{equation}
\end{prop}
%\begin{remark}\label{division points}
%For any $s=t_{i}$, the distance $d_s(\cdot,\cdot)$ in $\ref{e: heat kernel}$ and throughout the proof will always denote the distance with respect to the metric $g_{i-1}(t_i)$. Recall in the definition, $g_i(t_i)$ is the final metric of the Ricci flow $(M_{i-1},g_{i-1}(t_i))$. Also, throughout the proof, every integration at $s=t_i$ is over the manifold $M_{i+1}$ if we don't specify. %have to distinguish the manifolds and metrics before and after time $\tau_j$ when it is a division time. For this purpose, we use upper index $1$, respectively $2$ to represent them. To be specific, suppose $1=\tau_0>\tau_1>\cdots>\tau_j>\tau_{j+1}>\cdots$, during each $[\tau_{j+1},\tau_{j}]$ the Ricci flow is smooth and complete. Then $M^1_j$ and $g^1(\tau_j)$ denote the manifold and the final metric corresponding to the complete Ricci flow on $[\tau_{j+1},\tau_j]$. And $M^2_j$ and $g^2(\tau_j)$ are the manifold and the initial metric corresponding to the complete Ricci flow on $[\tau_{j},\tau_{j-1}]$. 
%\end{remark}

\begin{remark}\label{l: rescaling}
It may seem surprising that it is not necessary to assume the equality of metrics $g_{j}(t_{j+1})$ and $g_{j+1}(t_{j+1})$ on $M_{j+1}$. But as we will see in the proof below, the expanding condition $g_{j}(t_{j+1})\le g_{j+1}(t_{j+1})$ is compatible with the application of the Shrinking Lemma and hence sufficient for us to get the conclusion. In later application to the proof of Theorem $\ref{t: theorem1}$, the metric $g_{j+1}(t_{j+1})$ is the conformally changed metric of $g_j(t_{j+1})$, which is not less than $g_j(t_{j+1})$ everywhere on $M_{j+1}$, and agrees with it on a smaller region.  
\end{remark}

\begin{proof}
For notational convenience, the same letter $C$ will be used to denote constants depending on $n$, $A$ and $\nu$.

\textbf{Part 1}
Let us first establish the estimate $\eqref{e: heat kernel}$ for $t=t_{k+i}$ and $s=t_i$ for some $i$ and $k$. Rescaling the flow $g(t), t\in[t_{i-1},t_{k+i}]$ to $\hat{g}(\tau)=\frac{1}{t_{k+i}-t_{i-1}}g(\tau(t_{k+i}-t_{i-1})+t_{i-1}), \tau\in[0,1]$, the ``expanding time" sequence 
$$t_{k+i}>t_{k+i-1}>\cdots>t_{k+i-j}>\cdots>t_{i+1}>t_i>t_{i-1}$$ becomes 
$$1=\tau_0>\tau_1>\cdots>\tau_j>\cdots>\tau_{k-1}>\tau_k>\tau_{k+1}=0$$
where $\tau_j:=\frac{t_{k+i-j}-t_{i-1}}{t_{k+i}-t_{i-1}}=\frac{\nu^{k-j+1}-1}{\nu^{k+1}-1}$, for $j=0,1,2,...,k+1$. Then for each $j$, we have
\begin{equation}\label{good}
\tau_{j}-\tau_{j+1}=\frac{\nu^{k-j+1}-\nu^{k-j}}{\nu^{k+1}-1}\le\nu^{-j}.
\end{equation}
%First, we prove $\eqref{e: heat kernel}$ for $t=t_{k+i}$ and $s=t_i$. We do a parabolic rescaling and time shifting to the Ricci flow on $[t_{i},t_{k+i}]$, i.e. we define $\hat{g}(\tau)=\frac{1}{t_{k+i}-t_i}g(\tau(t_{k+i}-t_i)+t_i)$, $\tau\in[0,1]$. It's easy to verify that under the scaling,
%\begin{equation}
%G(x,t;y,s)=\frac{1}{(t_{k+i}-t_i)^{\frac{n}{2}}}\hat{G}(x,\frac{t-t_i}{t_{k+i}-t_i};y,\frac{s-t_i}{t_{k+i}-t_i})
%\end{equation}
%In particular, $G(x,t_{k+i};y,t_i)=\frac{1}{(t_{k+i}-t_i)^{\frac{n}{2}}}\hat{G}(x,1;y,0)$. By abusing the notation, we still denote the heat kernel after scaling by $G$. 
To show $\eqref{e: heat kernel}$ for $t=t_{k+i}$ and $s=t_i$, it's equivalent to show the following inequality under the new flow: 
\begin{equation}\label{that}
G(x,1;y,\tau_k)\le C exp\left(-\frac{d_{\tau_{k}^+}^2(x,y)}{C}\right).
\end{equation}
We note that by Remark $\ref{l: rescaling}$, the new flow $\hat{g}(\tau)$ satisfies the curvature and volume conditions in $\eqref{trader}$.

Since $\tau_1\le\nu^{-1}$, applying the Gaussian bound $\eqref{lotus}$ for standard heat kernel we find that
\begin{equation}\label{well}
G(x,1;\cdot,\tau_1)\le\frac{C}{(1-\tau_1)^{\frac{n}{2}}}\le C_0:=\frac{C}{(1-\nu^{-1})^{\frac{n}{2}}}.
\end{equation}
Let $C_0$ be fixed hereafter. Suppose by induction that $G(x,1;\cdot,\tau_j)\le C_0$ for some $j\ge 1$. Then for any $z$ such that $G(x,1;z,\tau_{j+1})$ is well defined, we have
\begin{equation}
\begin{split}
G(x,1;z,\tau_{j+1})&=\int_{M_{\tau_j^+}} G(x,1;w,\tau_j)G(w,\tau_j;z,\tau_{j+1})d_{\tau_j^-}w\\
&\le C_0\int_{M_{\tau_j^+}}G(w,\tau_j;z,\tau_{j+1})d_{\tau_j^-}w\le C_0
\end{split}
\end{equation}
where we used $\eqref{e: crucial2}$ in the last inequality.
So by induction we obtain 
\begin{equation}\label{cake}
G(x,1;\cdot,\tau_j)\le C_0,
\end{equation}
for all $j=1,2,...,k$. In particular, we have 
$G(x,1;\cdot,\tau_k)\le C_0$.
This implies $\eqref{that}$ when $d_{\tau_{k}^+}(x,y)$ is controlled. So it remains to show $G(x,1;y,\tau_k)\le\ exp\left(-\frac{d^2}{C}\right)$ whenever $d_{\tau_k^+}(x,y)\ge 4d(1-(\sqrt[4]{\nu})^{-1})$ for a large number $d$ (which we will specify in the course of proof).
For each $j=1,2,...,k$, let 
\begin{equation}
r_j=4d(1-(\sqrt[4]{\nu})^{-j}).
\end{equation}
Then set $B_j=B_{\tau_j^+}(x,r_j)$, $C_j=M_{\tau_j^+}-B_j$ and
\begin{equation}
a_j:=\sup_{C_j} G(x,1;\cdot,\tau_j). 
\end{equation}
Then it suffices to show the following Claim:

\begin{claim}\label{yes}
$a_{j}\le C exp(-\frac{d^2}{C})$, for some constant $C$ which is uniform for all $j=1,2,...,k$.
\end{claim}

\begin{proof}[Proof of Claim \ref{yes}]
For each $j$, the expanding condition $g_{j-1}(t_{j})\le g_{j}(t_j)$ implies $B_j=B_{\tau_j^+}(x,r_j)\subset B_{\tau_j^-}(x,r_j)$. Applying the Shrinking Lemma on $[\tau_{j+1},\tau_j]$, we find that $B_{\tau_j^-}(x,r_j)\subset B_{\tau_{j+1}^+}(x,r_j+\beta\sqrt{A}\sqrt{\tau_j-\tau_{j+1}})$. Thus for any $z\in C_{j+1}$ and $w\in B_j$, the triangle inequality implies 
\begin{equation}\label{dog}
d_{\tau_{j+1}^+}(z,w)\ge r_{j+1}-r_{j}-\beta \sqrt{A}\sqrt{\tau_j-\tau_{j+1}}.
\end{equation}
By $\eqref{good}$, $\sqrt{\tau_j-\tau_{j+1}}\le(\sqrt{\nu})^{-j}\le(\sqrt[4]{\nu})^{-j}$, we choose 
$$d\ge \frac{\beta\sqrt{A}}{2(1-(\sqrt[4]{\nu})^{-1})},$$ then $\eqref{dog}$ gives 
\begin{equation}\label{Deltar_j}
d_{\tau_{j+1}^+}(z,w)\ge \delta r_j:=\frac{2d(1-(\sqrt[4]{\nu})^{-1})}{(\sqrt[4]{\nu})^{j}}.
\end{equation}
To conclude, we have
\begin{equation}\label{pear}
B_j\subset M_{\tau_{j+1}^+}-B_{\tau_{j+1}^+}(z,\delta r_j).   
\end{equation}

By Definition $\ref{gene}$, we have
\begin{equation}
G(x,1;z,\tau_{j+1})=\int_{M_{\tau_j^+}} G(x,1;w,\tau_{j})G(w,\tau_j;z,\tau_{j+1})\, d_{\tau_j^-} w,
\end{equation}
for any $z\in C_{j+1}$ fixed. We split the following integral $\mathcal{I}[M_{\tau_j^+}]:=G(x,1;z,\tau_{j+1})$ into the integrals over $C_j$ and $B_j$. We obtain from the definition of $a_j$ and 
%$\int_{M_{\tau_j^+}} G(w,\tau_j;z,\tau_{j+1})\,d_{\tau_j^-}w\le 1$ from 
$\eqref{e: crucial2}$ that

\begin{equation}\label{C}
\begin{split}
\mathcal{I}[C_j]&=\int_{C_j} G(x,1;w,\tau_j)G(w,\tau_j;z,\tau_{j+1}) d_{\tau_j^-} w\\
&\le a_j\int_{M_{\tau_j^+}} G(w,\tau_j;z,\tau_{j+1})d_{\tau_j^-}w,\le a_j.
\end{split}
\end{equation}
To estimate $\mathcal{I}[B_j]$, we notice that by $\eqref{pear}$, $\eqref{cake}$ and $\eqref{trader}$ we have
\begin{equation*}
\begin{split}
\mathcal{I}[B_j]&\le C_0\int_{B_j} G(w,\tau_{j};z,\tau_{j+1})\,\,d_{\tau_{j}^-}w,\\
&\le C_0\int_{ M_{\tau_{j+1}^+}-B_{\tau_{j+1}^+}(z,\delta r_j) } G(w,\tau_{j};z,\tau_{j+1})\,\,d_{\tau_{j}^-}w\\
&\le C\int_{ M_{\tau_{j+1}^+}-B_{\tau_{j+1}^+}(z,\delta r_j) } G(w,\tau_{j};z,\tau_{j+1})\,\,d_{\tau_{j+1}^+}w.
\end{split}
\end{equation*}
Then applying the Gaussian bound $\eqref{lotus}$ to $G(w,\tau_{j};z,\tau_{j+1})$ and calculating as Lemma $\ref{calculation: partial integral}$ we have
\begin{equation}\label{B}
\begin{split}
 \mathcal{I}[B_j]&\le C exp\left(-\frac{(\delta r_j)^2}{C(\tau_j-\tau_{j+1})}\right).
\end{split}
\end{equation}
Plugging $\eqref{good}$ and $\eqref{Deltar_j}$ into $\eqref{B}$ we have
\begin{equation}
\mathcal{I}[B_j]\le C exp\left(-\frac{(\sqrt{\nu})^j d^2}{C}\right).
\end{equation}
Combining $\eqref{C}$ and $\eqref{B}$, we see that $G(x,1;z,\tau_{j+1})\le a_j+C exp(-\frac{d^2(\sqrt{\nu}^j)}{C})$ for arbitrary $z$ in $C_{j+1}$. Hence by the definition of $a_{j+1}$, there holds 
\begin{equation*}
\begin{split}
a_{j+1}&\le a_j+C exp\left(-\frac{d^2(\sqrt{\nu})^j}{C}\right)\le a_1+C \sum\displaylimits_{l=1}^j exp\left(-\frac{d^2(\sqrt{\nu})^l}{C}\right)\\
&\le a_1+C exp\left(-\frac{d^2}{C}\right)\sum\displaylimits_{l=1}^j exp\left(-\frac{d^2((\sqrt{\nu})^l-1)}{C}\right)\\
&\le a_1 + C\,exp\left(-\frac{d^2}{C}\right).
\end{split}
\end{equation*}

Note $a_1=\sup_{C_1} G(x,1;\cdot,\tau_1)$. For any $z\in C_1=M_{\tau_1^+}-B_{\tau_1^+}(x,r_1)$, we have $d_{\tau_1^+}(x,z)\ge r_1=4d(1-(\sqrt[4]{\nu})^{-1})$. Substituting this into the ordinary Gaussian bound, we get
$G(x,1;z,\tau_1)\le\frac{C}{(1-\tau_1)^{\frac{n}{2}}}exp(-\frac{d^2}{C(1-\tau_1)})$.
This gives $a_1\le C exp(-\frac{d^2}{C})$.
Hence $a_{j+1}\le C exp(-\frac{d^2}{C})$. This finishes the proof of the claim.
\end{proof}

To summarize, we showed for $t=t_{k+i},s=t_i$ and $x,y$ such that $G(x,t_{k+i};y,t_i)$ is defined, we have the Gaussian bound. 
\begin{equation}\label{e: division times}
G(x,t_{k+i};y,t_i)\le \frac{C}{(t_{k+i}-t_i)^{\frac{n}{2}}}exp\left(-\frac{d^2_{t_i^+}(x,y)}{C(t_{k+i}-t_i)}\right).
\end{equation}
We will use this to derive the Gaussian bound $\eqref{e: heat kernel}$ for arbitrary $t$ and $s$.

\textbf{Part 2}
To show $\eqref{e: heat kernel}$ for arbitrary $t$ and $s$, there are two cases left. The first is that neither $t$ nor $s$ is expanding time, and the second is that one of them is an expanding time. Since the second case follows a same but easier route than the first one, we prove the first case below. 

Since $t$ and $s$ are not expanding times, we may assume $t\in(t_{k+i},t_{k+i+1})$ and $s\in(t_{i},t_{i+1})$ for some $k$ and $i$. Rescaling the flow on $[s,t]$ to a new flow on $[0,1]$, for the same reason as in Part 1, it suffices to show for any very large $d$ (which we specify below) and $x,y$ such that $d_0(x,y)\ge 5d$, we have
\begin{equation}\label{sleep}
G(x,1;y,0)\le C exp\left(-\frac{d^2}{C}\right).
\end{equation}
Under rescaling, $t_{k+i}$ and $t_{i+1}$ become $\tau_2:=\frac{t_{k+i}-s}{t-s}$ and $\tau_1:=\frac{t_{i+1}-s}{t-s}$, respectively. By Definition $\ref{gene}$ of the generalized heat kernel, we have
\begin{equation}
G(x,1;y,0)=\int_{M_{\tau_2^+}}\int_{M_{\tau_1^+}}\,G(x,1;z,\tau_{2})G(z,\tau_{2};w,\tau_{1})G(w,\tau_{1};y,0)d_{\tau_1^-}w\, d_{\tau_2^-}z.
\end{equation}
We split the integral $\mathcal{I}[M_{\tau_2^+}\times M_{\tau_1^+}]:=G(x,1;y,0)$ over three regions 
\begin{equation}
\begin{split}
U&=\{(z,w)\,|\,z\in B_{\tau_2^-}(x,d) \;and\; w\in B_{\tau_1^-}(y,d)\},\\
V&=\{(z,w)\,|\,z\notin B_{\tau_2^-}(x,d)\},\\
W&=\{(z,w)\,|\,w\notin B_{\tau_1^-}(y,d)\}.
\end{split}
\end{equation}
Then $G(x,1;y,0)\le \mathcal{I}[U]+\mathcal{I}[V]+\mathcal{I}[W]$. Since $\tau_1$ and $\tau_2$ are both expanding times and $\tau_2-\tau_1$ is bounded below by a positive number depending only on $\nu$, the result from Part 1 implies
\begin{equation}\label{tau_1tau_2}
 G(z,\tau_2;w,\tau_1)\le C\,exp\left(-\frac{d_{\tau_1^+}^2(z,w)}{C}\right).
\end{equation}
If we choose $d\ge\beta\sqrt{A}$, then for any $z\in B_{\tau_2^-}(x,d)$ and $w\in B_{\tau_1^-}(y,d)$, the Shrinking Lemma together with the expanding conditions imply $d_{\tau_1^-}(x,z)\le d_{\tau_1^+}(x,z)\le d_{\tau_2^-}(x,z)+\beta\sqrt{A}\le d_{\tau_2^-}(x,z)+d\le 2d$, and $d_{\tau_1^-}(x,y)\ge d_0(x,y)-\beta\sqrt{A}\ge 4d$. Then by triangle inequality we have 
\begin{equation}\label{dd}
\begin{split}
 d_{\tau_1^+}(z,w)&\ge d_{\tau_1^-}(z,w)
 \ge d_{\tau_1^-}(x,y)-d_{\tau_1^-}(x,z)-d_{\tau_1^-}(y,w)
 \ge d.
 \end{split}
 \end{equation}
Hence by $\eqref{tau_1tau_2}$, $\eqref{dd}$, $\eqref{lotus}$ and $\eqref{e: crucial2}$ we have
\begin{equation}\label{e: I}
\begin{split}
\mathcal{I}[U]&\le C\,exp\left(-\frac{d^2}{C}\right)\cdot\left(\int_{M_{\tau_2^+}} G(x,1;z,\tau_2)\,d_{\tau_2^-}z\right)\cdot\left(\int_{M_{\tau_1^+}} G(w,\tau_1;y,0)\,d_{\tau_1^-}w\right)\\
&\le C\,exp\left(-\frac{d^2}{C}\right)\cdot C\cdot 1=C\,exp\left(-\frac{d^2}{C}\right).
\end{split}
\end{equation}
%To estimate $\mathcal{J}$, we notice  $\sup_{z,w}G(z,\tau_2;w,\tau_1)\le C$ and $\int_{M_{\tau_1^+}} G(w,\tau_1;y,0)\,d_{\tau_1^-}w\le 1$, so
And $\eqref{tau_1tau_2}$, $\eqref{e: crucial2}$, $\eqref{lotus}$ together with Lemma $\ref{calculation: partial integral}$ imply
\begin{equation}\label{e: J}
\begin{split}
\mathcal{I}[V]
%&\le C \left(\int_{z\notin B_{\tau_2^-}(x,d)} G(x,1;z,\tau_2)\,\,d_{\tau_2^-}z\right)\cdot \left(\int_{M_{\tau_1^+}} G(w,\tau_1;y,0)\,d_{\tau_1^-}w\right)\\
&\le C \left(\int_{z\notin B_{\tau_2^-}(x,d)} G(x,1;z,\tau_2)\,\,d_{\tau_2^-}z\right)\le C exp\left(-\frac{d^2}{C}\right).
%&\le \frac{C}{(1-\tau_2)^{\frac{n}{2}}}\int_{z\notin B_{\tau_2^-}(x,d)}exp\left(-\frac{d^2_{\tau_2^+}(x,z)}{C(1-\tau_2)}\right)\,\,d_{\tau_2^-}z\\, 
\end{split}
\end{equation}
Similarly we have
\begin{equation}\label{e: K}
\mathcal{I}[W]\le C exp\left(-\frac{d^2}{C}\right).
\end{equation}
So $\eqref{sleep}$ follows from $\eqref{e: I},\eqref{e: J},\eqref{e: K}$ immediately.
\end{proof}
\end{subsection}

\begin{subsection}{Gradient of heat kernel}
In this subsection, we use Proposition $\ref{generalized heat kernel}$ to derive an upper bound for the gradient of the generalized heat kernel. Assume all the conditions are the same as in Proposition $\ref{generalized heat kernel}$. We choose and fix some $x\in M_{i}$ , $t\in(t_{i},t_{i+1}]$ for some $i$. Then $G(x,t;\cdot,\cdot)$ is a solution to the heat equation $\frac{\partial}{\partial s'} G(x,t;z,s')+\Delta_{z,s'}G(x,t;z,s')=0$ on $M_j\times(t_j,\min(t_{j+1},t)]$, $j\le i$. 
%Scaling the flow on $[t_j,s]$ to $\hat{g}(\tau):=\frac{1}{s-t_j}g(\tau(s-t_j)+t_j)$, we have a solution $\hat{g}(\tau)$ on $[0,1]$ with $|\Rm|_{\hat{g}(\tau)}\le A$. Scale also $G(x,t;z,s')$ to $u(z,\tau):=G(x,t;z,(s-t_j)\tau+t_j)$. Then $u(z,\tau)$ satisfies $\pt u(z,\tau)+\Delta_{\hat{g}(\tau)}u(z,\tau)=0$. 
For an arbitrary $(y,s)\in M_{j}\times (t_j,\min(t_{j+1},t)]$, $j\le i$, applying the standard result of Schauder estimate (see \cite{trudinger} for example), we see that there is a constant $C$ depending on $A$ and $n$ such that
%\begin{equation}
 %|\nabla u|_{\hat{g}(1)}(y)\le C\,\sup\,u 
%\end{equation}
%where the supremum is taken over $B_{\hat{g}(1)}(y,1)\times[0,1]$. Scaling back, we have
\begin{equation}
|\nabla G|(x,t;y,s)\le\frac{C}{\sqrt{s-t_j}}\,\sup\,G(x,t;\cdot,\cdot),
\end{equation}
where the supremum is taken over $B_{g(s)}(y,\sqrt{s-t_j})\times[t_j,s]$.

Since $|\Rm|\le\frac{A}{t}$ on $M_j\times [t_j,t_{j+1}]$,we have a constant $C_1=C_1(n,A,\nu)>0$ such that for any $s,s'\in[t_j,t_{j+1}]$, $C_1^{-1}d_{s'}\le d_s\le C_1\,d_{s'}$. Suppose $d_{s}(x,y)\ge d$ for a large number $d$ satisfying
\begin{equation}\label{35}
 d\ge 2C_1(\sqrt{t_{i+1}-t_{i}}+\beta\sqrt{A}\sqrt{t_{i+1}-t_i}).
\end{equation}
We claim the following Gaussian bound of $|\nabla G|(x,t;y,s)$:

\begin{claim}\label{claim: gradient of G}
\begin{equation}\label{lily}
|\nabla G|(x,t;y,s)\le\frac{1}{\sqrt{s-t_j}}\frac{C}{t_{i+1}^{\frac{n}{2}}}exp\left(-\frac{d_{s}^2(x,y)}{Ct_{i+1}}\right)
\end{equation}
for some constant $C$ that only depends on $A$, $\nu$ and $n$.
\end{claim}

\begin{proof}[Proof of Claim \ref{claim: gradient of G}]
For any $(z,s')\in B_{g(s)}(y,\sqrt{s-t_{j}})\times[t_{j},s]$, first we have by the Shrinking Lemma that $d_{s'}(y,z)\le d_s(y,z)+\beta\sqrt{A}(\sqrt{s-s'})$. Then the triangle inequality and $\eqref{35}$ we get
\begin{equation}
\begin{split}
d_{s'}(x,z)&\ge d_{s'}(x,y)-d_{s'}(y,z)\\
&\ge d_{s'}(x,y)-d_s(y,z)-\beta\sqrt{A}(\sqrt{s-s'})\\
&\ge d_{s'}(x,y)-\sqrt{t_{j+1}-t_{j}}-\beta\sqrt{A}\sqrt{t_{j+1}-t_j}\\
&\ge C_1^{-1}d_s(x,y)-\sqrt{t_{j+1}-t_{j}}-\beta\sqrt{A}\sqrt{t_{j+1}-t_j}\\
&\ge \tfrac{1}{2}C_1^{-1}d_s(x,y).
\end{split}
\end{equation}
So by Proposition $\ref{generalized heat kernel}$ we have 
\begin{equation}\label{lemon}
\begin{split}
G(x,t;z,s')&\le\frac{C}{(t-s')^{\frac{n}{2}}}exp\left(-\frac{d_{s'}^2(x,z)}{C(t-s')}\right)\le\frac{C}{(t-s')^{\frac{n}{2}}}exp\left(-\frac{d_{s}^2(x,y)}{C(t-s')}\right).
\end{split}
\end{equation}
Since $d_s(x,y)\ge d$ and $t-s'\le t_{i+1}$, Lemma $\ref{calculation: inequity}$ implies
\begin{equation}
G(x,t;z,s')\le \frac{C}{(t_{i+1})^{\frac{n}{2}}}exp\left(-\frac{d_{s}^2(x,y)}{Ct_{i+1}}\right). 
\end{equation}
The claim thus follows by letting $(z,s')$ run over $B_{g(s)}(y,\sqrt{s-t_j})\times[t_j,s]$.
\end{proof}

\end{subsection}

\end{section}

%%begin applying the curvature decay lemma in the following scheme:
%%first, improve the constant from C_3 to C_1
%%second, do the conformal change then the constant changes from C_1 to C_2 but at the same time become global.
%%third, run a new and complete RF then the constant increases from C_2 to C_3.

\begin{section}{Proof of Theorem $\ref{t: theorem1}$}
%We now start to prove Theorem $\ref{t: theorem1}$. As we said in introduction, we extend the Ricci flow by induction. Assume after the $i$-th inductive step, we have a Ricci flow in expansion as we defined in Section 5, satisfying three assumptions which we will state later. We want to show that these assumptions guarantee the $(i+1)$-th extension which generates a new Ricci flow satisfying the same assumptions.
%Our goal is to show that after each induction step, our Ricci flow satisfies three assumptions. These assumptions would in turn make sure that we can start another extension and get a Ricci flow satisfying the same assumptions. 
%The Curvature Decay Lemma from Section 3 plays a key role in verifying the assumptions. In the current section, we perform the $(i+1)$-th extension and verify the first two assumptions. The verification of the last assumption will be left to the next section.

First, we consider the conditions given in Theorem $\ref{t: theorem1}$. 
The upper bound on $\ell(x,0)$ implies a lower bound on Ricci curvature, that is, $\ell(x,0)\le\alpha_0\le 1$ implies $\textnormal{Ric}\ge -K(n)$. So by Bishop-Gromov comparison, reducing $v_0$ to a smaller positive number depending only on the original $v_0$ and $n$, we may assume without loss of generality that 
\begin{equation}\label{e: volume}
Vol_{g(0)}B_{g(0)}(x,r)\ge v_0 r^n
\end{equation}
for all $x\in B_{g(0)}(x_0,s_0-1)$ and $r\in(0,1]$.
We can also assume $\alpha_0$ without loss of generality that
\begin{equation}\label{e: assump on initial u}
\alpha_0\le\frac{1}{2C_4}<1
\end{equation}
where $C_4=C_4(v_0,n)>2$ is to be determined later. Otherwise, we get the result by applying the above result to a rescaled metric and then scale it back.

By the relative compactness of $B_{g(0)}(x_0,s_0)$, there exists some $\rho\in(0,\frac{1}{2}]$ such that $|\textnormal{Rm}|\le\frac{1}{\rho^2}$, $B_{g(0)}(x,\rho)\subset\subset M$ and $\textnormal{inj}_{g(0)}(x)\ge\rho$ for all $x\in B_{g(0)}(x_0,s_0)$. The constant $\rho$ may depend on $(M,g(0))$, $x_0$ and $s_0$. By applying Lemma $\ref{l: conformal}$, with $U:=B_{g(0)}(x_0,s_0)$, we can find a connected subset $\tilde{M}\subset U\subset M$ containing $B_{g(0)}(x_0,s_0-\frac{1}{2})$, and a smooth, complete metric $\tilde{g}(0)$ on $\tilde{M}$ with $\sup\displaylimits_{\tilde{M}}|\textnormal{Rm}|_{\tilde{g}(0)}<\infty$ such that on $B_{g(0)}(x_0,r_0)$, where $r_0:=s_0-1>3$, the metric remains unchanged. Taking Shi's Ricci flow we get a smooth, complete, bounded-curvature Ricci flow $g_0(t)$ on $M_0:=\tilde{M}$, existing for some nontrivial time interval $[0,t_1]$. In view of the boundedness of the curvature, after possibly reducing $t_1$ to a smaller positive value, we may trivially assume that $|\textnormal{Rm}|_{g(t)}\le\frac{C_3}{t}$ for all $t\in (0,t_1]$ and $\ell(x,t)\le 2\alpha_0 <1$ for all $x\in B_{g(0)}(x_0,r_0)$ and  $t\in[0,t_1]$. The constant $C_3=C_3(v_0,n)$ will be given below.

Of course, our flow still lacks a uniform control on its existence time. Below we will carry out an inductive argument to show that $t_1$ could be extended up to a uniform time $t_k$, while the repeating time $k$ may be allowed to depend on $(M,g)$. 

Now we begin the proof of Theorem $\ref{t: theorem1}$. First, suppose we have constructed a Ricci flow in expansion $(\{M_j\}_{j=1}^i,\{g_j(t)\}_{j=1}^{i},\nu)$ with $(M_0,g_0(t))_{t\in[0,t_1]}$ as above. Suppose further the Ricci flow in expansion satisfies the following a priori assumptions:
\begin{description}
\item[(APA 1)]Restricting it on $B_{g(0)}(x_0,r_i)$, we get a smooth Ricci flow $g(t)$ up to $t_{i+1}$;
\item[(APA 2)]For each complete Rici flow $(M_j,g_j(t))$, we have $|\Rm|_{g_j(t)}\le\frac{C_3}{t}$;
\item[(APA 3)]$\ell(x,t)\le C_4\alpha_0<1$ for all $t\in[0,t_{i+1}]$ and $x\in B_{g(0)}(x_0,r_i)$.
\end{description}
where the constants $C_3, C_4, \nu$ depending on $v_0, n$ will be specified in the course of the proof. 

Our goal is to extend it to a new Ricci flow in expansion $(\{M_j\}_{j=1}^{i+1},\{g_j(t)\}_{j=1}^{i+1},\nu)$ by adding a complete Ricci flow $(M_{i+1},g_{i+1}(t))$ piece existing for $[t_{i+1},t_{i+2}]$, and show that it still satisfies (APA 1)-(APA 3). In the current section, we construct $(M_{i+1}, g_{i+1}(t))$, and then verify (APA 1) and (APA 2), and we leave the verification of (APA 3) to the next section.

%For the given $v_0$ and $K=K(n)$, where $-K(n)$ is the Ricci curvature corresponding to $\ell=1$. 
Let $C_1\ge 1$ and $\tilde{T}>0$ be the constants from the Curvature Decay Lemma (Lemma $\ref{l: curvature decay}$) when $K=1$ and $v_0=v_0$. With this choice of $C_1$, we set $C_2=\gamma C_1$ and $C_3=4C_2=4\gamma C_1>1$, where $\gamma=\gamma(n)\ge 1$ is the constant from the Conformal Change Lemma (Lemma $\ref{l: conformal}$), and set $\nu=1+\frac{1}{4C_3}$. Choose $\tau$ such that
\begin{equation}\label{e: constants conditions}
\tau\le\hat{T}, \quad\beta^2C_3\tau\le\frac{1}{16}\sqrt{\tau}\le 1, \quad\tau\le 1,\quad \tau\le\frac{C_1}{4} ,
\end{equation}
where $\beta\ge 1$ is the constant from the Shrinking Lemma. We can also assume that $2t_{i+1}\le\tau$, because otherwise we get the desired uniform existence time $\frac{\tau}{2}$.
%\begin{equation}
%t_{i+1}\le t_{i+2}=\nu t_{i+1}\le 2t_{i+1}\le\tau\le 1.
%\end{equation}

In the Claim below, we show that in fact we have a stronger curvature decay bound $|\Rm|_{g(t)}\le\frac{C_1}{t}$. However, the original curvature decay will nevertheless be used to control the distance distortion. 

\begin{claim}\label{shrink}
For all $x\in U:=B_{g(0)}(x_0,r_i-2\sqrt{\frac{t_{i+1}}{\tau}})$, we have $B_{g(t)}(x,\sqrt{\frac{t}{\tau}})\subset\subset B_{g(0)}(x_0,r_i)$, $\textnormal{inj}_{g(t)}(x)\ge\sqrt{\frac{t}{C_1}}$ and $|\textnormal{Rm}|_{g(t)}(x)\le\frac{C_1}{t}$, for all $t\in(0,t_{i+1}]$.
\end{claim}

\begin{proof}[Proof of Claim \ref{shrink}]
By the Shrinking Lemma, for any $x\in B_{g(0)}(x_0, r_i-2\sqrt{\frac{t_{i+1}}{\tau}})$, the triangle inequality implies that $B_{g(0)}(x,2\sqrt{\frac{t_{i+1}}{\tau}})\subset\subset B_{g(0)}(x_0,r_i)$ and hence by assumption (APA 3), $\ell(y,t)\le 1$ on $B_{g(0)}(x,2\sqrt{\frac{t_{i+1}}{\tau}})$ for all $t\in[0,t_{i+1}]$. Scaling the solution to $\hat{g}(t):=\frac{\tau}{t_{i+1}}g(t\frac{t_{i+1}}{\tau})$ we see that we have a solution $\hat{g}(t)$ on $B_{g(0)}(x_0,r_{i})\supset\supset B_{\hat{g}(0)}(x,2)$, $t\in[0,\tau]$ with $|\textnormal{Rm}|_{\hat{g}(t)}\le\frac{C_3}{t}$ and $\ell(\cdot,\cdot)\le 1$ on $B_{\hat{g}(0)}(x,2)\times(0,\tau]$.

On the one hand, applying the Shrinking Lemma to $\hat{g}(t)$, we find that $B_{\hat{g}(t)}(x,2-\beta\sqrt{C_3 t})\subset B_{\hat{g}(0)}(x,2)$ for all $t\in[0,\tau]$, and in particular $B_{\hat{g}(t)}(x,1)\subset B_{\hat{g}(0)}(x,2)$ because $\tau\le \frac{1}{\beta^2 C_3}$. Thus we have $\ell(\cdot,\cdot)\le 1$ on $\bigcup_{s\in[0,\tau]}B_{\hat{g}(s)}(x,1)\times[0,\tau]$. On the other hand, the volume inequality $\eqref{e: volume}$  transforms to $Vol_{\hat{g}(0)}B_{\hat{g}(0)}(x,1)\ge v_0$.

%%%%%%%%%%%% note here we need to generalize lemma 4.1 as well
Applying the Curvature Decay Lemma (Lemma $\ref{l: curvature decay}$) to $\hat{g}(t)$, we have $\textnormal{inj}_{\hat{g}(t)}(x)\ge\sqrt{\frac{t}{C_1}}$ and $|\textnormal{Rm}|_{\hat{g}(t)}(x)\le\frac{C_1}{t}$ for all $0<t\le\tau$. Scaling back, we see that $B_{g(t)}(x,\sqrt{\frac{t_{i+1}}{\tau}})\subset\subset B_{g(0)}(x_0,r_{i})$, $\textnormal{inj}_{g(t)}(x)\ge\sqrt{\frac{t}{C_1}}$ and $|\textnormal{Rm}|_{g(t)}(x)\le\frac{C_1}{t}$ for $t\in(0,t_{i+1}]$.

\end{proof}

Specializing the claim $\ref{shrink}$ to $t=t_{i+1}$, we have 
$|\textnormal{Rm}|_{g(t_{i+1})}(x)\le\frac{C_1}{t_{i+1}}$ and $\textnormal{inj}_{g(t_{i+1})}(x)\ge\sqrt{\frac{t_{i+1}}{C_1}}$
for any $x\in U:=B_{g(0)}(x_0,r_i-2\sqrt{\frac{t_{i+1}}{\tau}})$.
Now we apply the Conformal Change Lemma $\ref{l: conformal}$ with
$U=B_{g(0)}(x_0,r_i-2\sqrt{\frac{t_{i+1}}{\tau}})$, $N=B_{g(0)}(x_{0},r_{i})$, $g(t_{i+1})$ and $\rho^2:=\frac{t_{i+1}}{C_1}\le 1$, and obtain a new, possibly disconnected, smooth manifold $(\tilde{U},h)$, each component of which is complete, such that
\begin{enumerate}
 \item $ |\textnormal{Rm}|_{h}\le\gamma\frac{C_1}{t_{i+1}}=\frac{C_2}{t_{i+1}}$ and $\textnormal{inj}_h \ge\sqrt{\frac{t_{i+1}}{\gamma C_1}}=\sqrt{\frac{t_{i+1}}{C_2}}$ for all $x\in\tilde{U}$,
 \item $U_\rho\subset\tilde{U}\subset U$,
 \item $h=g(t_{i+1})$ on $\tilde{U}_\rho\supset U_{2\rho}$
\end{enumerate}
where $U_r=\{x\in U| B_g(x,r)\subset\subset U\}$.

\begin{claim}\label{this}
We have $B_{g(0)}(x_0,r_{i}-4\sqrt{\frac{t_{i+1}}{\tau}})\subset U_{2\rho}$ where the metric $g(t_{i+1})$ and $h$ agree.
\end{claim}

\begin{proof}[Proof of Claim \ref{this}]
By definition of $\rho$, $U$ and condition $\eqref{e: constants conditions}$, for  every $x\in B_{g(0)}(x_0,r_i-4\sqrt{\frac{t_{i+1}}{\tau}})$, the triangle inequality implies $B_{g(0)}(x,2\sqrt{\frac{t_{i+1}}{\tau}})\subset\subset U$. 
By (APA 2), we have $|\textnormal{Rm}|_{g(t)}\le\frac{C_3}{t}$ on $B_{g(0)}(x_0,r_{i})$, and hence on $B_{g(0)}(x, 2\sqrt{\frac{t_{i+1}}{\tau}})$ for all $t\in (0,t_{i+1}]$. Applying the Shrinking Lemma we have $B_{g(0)}(x,2\sqrt{\frac{t_{i+1}}{\tau}})\supset B_{g(t)}(x,2\sqrt{\frac{t_{i+1}}{\tau}}-\beta\sqrt{C_3 t})$ for all $t\in[0,t_{i+1}]$. Specializing to $t=t_{i+1}$ and use $\beta\sqrt{C_3 t_{i+1}}\le\sqrt{\frac{t_{i+1}}{\tau}}$
%and recalling that $\beta\sqrt{C_3 t_{i+1}}\le\sqrt{\frac{t_{i+1}}{\tau}}$ by the restriction on $\tau$, 
we see that $B_{g(t_{i+1})}(x,\sqrt{\frac{t_{i+1}}{\tau}})\subset\subset U$. By $\eqref{e: constants conditions}$ this gives $B_{g(t_{i+1})}(x,2\sqrt{\frac{t_{i+1}}{C_1}})\subset\subset U$ which means $x\in U_{2\rho}$.
\end{proof}

In view of Claim $\ref{this}$ we define the connected component of $(\tilde{U},h)$ that contains $B_{g(0)}(x_0,r_i-4\sqrt{\frac{t_{i+1}}{\tau}})$ as $M_{i+1}$. Then we restart the flow from $(M_{i+1},h)$ using Shi's complete bounded curvature Ricci flow.
By the doubling time estimate (Lemma $\ref{doubling}$), we have a complete Ricci flow $(M_{i+1},h(t))$ with $h(0)=h$ existing for $t\in[0,(\nu-1) t_{i+1}]$ and satisfying
\begin{equation}\label{e: C_2 and A_3}
|\textnormal{Rm}|_{h(t)}(y)\le 2\frac{C_2}{t_{i+1}} \;\;\;\;\textnormal{and}\;\;\;\;
Vol_{h(t)}B_{h(t)}(y,\sqrt{t_{i+1}})\ge\frac{t_{i+1}^{\frac{n}{2}}}{A_0} 
\end{equation}
for all $y\in M_{i+1}$, where $A_0$ is a constant depending on $C_2$ and thus on $v_0$ and $n$.
Setting $g_{i+1}(t)=h(t-t_{i+1})$ for $t\in[t_{i+1},t_{i+2}]=[t_{i+1},\nu t_{i+1}]$, we obtain a new Ricci flow in expansion $(\{M_j\}_{j=1}^{i+1},\{g_j(t)\}_{j=1}^{i+1},\nu)$, which clearly satisfies (APA 1). By $\eqref{e: C_2 and A_3}$ and $t_{i+2}=\nu t_{i+1}$ we have
\begin{equation}\label{Rm}
|\textnormal{Rm}|_{g(t)}(y)\le 2\frac{C_2}{t_{i+1}}\le\frac{C_3}{t}
\end{equation}
for all $t\in[t_{i+1},t_{i+2}]$. Hence we verified (APA 2). For the same reason, we have 
\begin{equation}\label{Vol}
Vol_{g(t)}B_{g(t)}(y,\sqrt{t})\ge\frac{t_{i+1}^{\frac{n}{2}}}{A_0}\ge\frac{t^{\frac{n}{2}}}{A}
\end{equation}
for all $t\in[t_{i+1},t_{i+2}]$, where $A=A_0\nu^{\frac{n}{2}}$ also depends on $v_0$ and $n$. The volume estimate is needed to apply Proposition $\ref{generalized heat kernel}$ in next section.
%For each time $t\in[t_{i+1},t_{i+2}]$, $(M_{i+1},g_{i+1}(t))$, $\frac{2C_2}{t_{i+1}}\le\frac{C_3}{t_{i+2}}\le\frac{C_3}{t}$ and $\frac{(\sqrt{t_{i+1}})^n}{A_0}\ge\frac{(\sqrt{t_{i+2}})^n}{(\sqrt{\nu})^n A_0}\ge\frac{(\sqrt{t})^n}{(\sqrt{\nu})^n A_0}$, thus we have
%\begin{equation}\label{Mi+1}
%|\textnormal{Rm}|_{g_{i+1}(t)}\le\frac{C_3}{t} \;\;\;\;\textnormal{and}\;\;\;\;
%Vol_{g_{i+1}(t)}B_{g_{i+1}(t)}(x,\sqrt{t})\ge\frac{(\sqrt{t})^n}{(\sqrt{\nu})^n A_0} 
%\end{equation}
%for all $x\in M_{i+1}$ and $t\in[t_{i+1},t_{i+2}]$.

%Thus, we have constructed a smooth extension to our Ricci flow containing $B_{g(0)}(x_0,r_{i}-4\sqrt{\frac{t_{i+2}}{\tau}})$, restricted on which we still call the flow by $g(t)$.
%It remains to check the second item in induction $I(i+1)$, that is, $\ell(x,t)\le C_4\alpha_0\le 1$ for all $x\in B_{g(0)}(x_0,r_{i+1})$ and $t\in[0,t_{i+2}]$ with some suitable $r_{i+1}$.

\end{section}

\begin{section}{Induction Step: Verification of (APA 3)}
In this section we finish the proof of Theorem $\ref{t: theorem1}$ by verifying (APA 3) for $(\{M_j\}_{j=1}^{i+1},\{g_j(t)\}_{j=1}^{i+1},\nu)$. More specifically, we determine  $r_{i+1}$ such that when restricted on $B_{g(0)}(x_0,r_{i+1})$, the smooth Ricci flow $g(t)$ satisfies $\ell(x,t)\le C_4\alpha_0<1$ for all $t\in[0,t_{i+2}]$. The estimates $\eqref{Rm}$ and $\eqref{Vol}$ allow us to apply Proposition $\ref{generalized heat kernel}$ to $(\{M_j\}_{j=1}^{i+1},\{g_j(t)\}_{j=1}^{i+1},\nu)$, and get the Gaussian bound for the generalized heat kernel $G(x,t;y,s)$: 
\begin{equation}
G(x,t;y,s)\le\frac{C}{(t-s)^{\frac{n}{2}}}exp(-\frac{d_{s^+}^2(x,y)}{C(t-s)}),
\end{equation}
where $C$ depends on $v_0$ and $n$. We will frequently use this inequality implicitly in this section. Also for notational convenience, the same letter $C$ will be used to denote positive constants depending on $n$ and $v_0$. We divide the integration estimates of $\ell$ into two steps. 

\textbf{Step 1}
\quad We derive a rough bound for $\ell$. Specifically, we show that $\ell$ is bounded above by a constant depending only on $v_0$ and $n$. This bound gives a lower bound for Ricci curvature with the same dependence, which will be used in the second step.
\begin{claim}\label{rough bound}
For any $(x,t)\in B_{g(0)}(x_0,r_{i}-4\sqrt{\frac{t_{i+1}}{\tau}}-\sqrt[4]{t_{i+2}})\times[0,t_{i+2}]$, we have $\ell(x,t)\le C$ and correspondingly $\Ric\ge -K$, where both $C$ and $K$ are positive constants depending only on $v_0$ and $n$. 
\end{claim}

\begin{proof}
%Throughout the proof of Lemma $\ref{rough bound}$, we choose and fix some $(x,t)\in B_{g(0)}(x_0,r_{i}-4\sqrt{\frac{t_{i+1}}{\tau}}-\sqrt[4]{t_{i+2}})\times[0,t_{i+2}]$. 
Since $(\{M_j\}_{j=1}^{i},\{g_j(t)\}_{j=1}^{i},\nu)$ satisfies (APA 3), we have $\ell(\cdot,\cdot)\le 1$ on $B_{g(0)}(x_0,r_i)\times[0,t_{i+1}]$. Thus it only remains to show $\ell(x,t)\le C$ for $t\in[t_{i+1},t_{i+2}]$. Recall the evolution inequality of $\ell$. 
\begin{equation}\label{e: u_original}
\pt \ell(x,t)\le\Delta \ell(x,t)+\textnormal{scal}(x,t)\ell(x,t)+C(n)\ell^2(x,t).
\end{equation}
Using the curvature decay $|\textnormal{Rm}|_{g(t)}\le\frac{C_3}{t}$ we obtained in Section 6, we have $\ell(x,t)\le\frac{C}{t_{i+1}}$ for all $(x,t)\in M_{i+1}\times[t_{i+1},t_{i+2}]$. Substituting this into $\eqref{e: u_original}$, we get	
\begin{equation}\label{e: u_rough}
\frac{\partial}{\partial t}\ell\le\Delta \ell+\textnormal{scal}\,\ell+C\ell^2\le\Delta\ell +\textnormal{scal}\ell+\frac{C}{t_{i+1}}\ell
\end{equation}
in the barrier sense. For any $t\in[t_{i+1},t_{i+2}]$, set $\mathcal{L}(x,t)=\ell(x,t)e^{-\frac{C}{t_{i+1}}t}$. Then $\mathcal{L}(x,t_{i+1})=\ell(x,t_{i+1})e^{-C}\le e^{-C}$ and 
\begin{equation}\label{e: U}
\frac{\partial}{\partial t}\mathcal{L}\le \Delta \mathcal{L}+ \textnormal{scal}\,\mathcal{L}
\end{equation}
in the barrier sense. Let $h(x,t)=\int_{M_{i+1}} G(x,t;z,t_{i+1})\mathcal{L}(z,t_{i+1})d_{t_{i+1}}z$, then $h$ solves the following initial value problem:
\begin{equation}
\ps h=\Delta h+\scal\,h \;\;\;\textnormal{and}\;\;\;h(\cdot,t_{i+1})=\LL(\cdot,t_{i+1}).
\end{equation}
By the maximum principle, we have
\begin{equation}
\mathcal{L}(x,t)\le h(x,t)=\mathcal{I}[M_{i+1}]:=\int_{M_{i+1}} G(x,t;y,t_{i+1})\mathcal{L}(y,t_{i+1})\,d_{t_{i+1}}y
\end{equation}
%$\mathcal{L}(x,t)\le h(x,t)$
for all $x\in M_{i+1}$ and $t\in[t_{i+1},t_{i+2}]$. 

Seeing that $M_i\supset B_{g(0)}(x_0,r_{i}-4\sqrt{\frac{t_{i+1}}{\tau}})$ where the local flow is smooth from $t=0$, we fix some $x\in B_{g(0)}(x_0,r_{i}-4\sqrt{\frac{t_{i+1}}{\tau}}-\sqrt[4]{t_{i+2}})$ and split the integral $\mathcal{I}[M_{i+1}]$ into two integrals over $\mathcal{B}_{i+1}:=B_{g(0)}(x,\sqrt[4]{t_{i+1}})$ and $\mathcal{C}_{i+1}:=M_{i+1}-B_{g(0)}(x,\sqrt[4]{t_{i+1}})$. %It then suffices to show $\mathcal{I}[B_1]$ and $\mathcal{I}[B_2]$ are bounded by some $C=C(v_0,n)$. 
Since $\mathcal{B}_{i+1}\subset B_{g(0)}(x_0,r_i-4\sqrt{\frac{t_{i+1}}{\tau}})$ where $\LL(\cdot,t_{i+1})\le\ell(\cdot,t_{i+1})\le 1$ by (APA 3), we can estimate
\begin{equation}\label{seven}
\mathcal{I}[\mathcal{B}_{i+1}]\le \int\displaylimits_{M_{i+1}} G(x,t;y,s)\,d_s y\le C.
\end{equation}
%Let $\psi(s)=\int\displaylimits_{M_{i+1}} G(x,t;y,s)\,d_s y$ for $s\in[t_{i+1},t)$. Since $(\{M_j\}_{j=1}^{i+1},\{g_j(t)\}_{j=1}^{i+1},\nu)$ satisfies (APA 2) which we established in last section, we have $\scal_{g(s)}(y)\le\frac{n(n-1)C_3}{t_{i+1}}$ for all $y\in M_{i+1}$ and hence
%\begin{equation*}
%\left|\ps\psi(s)\right|=\left|\int_{M_{i+1}} (-\Delta G-\textnormal{scal}\,G)\,d_s y\right|=\left|\int_{M_{i+1}} \textnormal{scal}\,G \,d_s y\right|\le\frac{n(n-1)C_3}{t_{i+1}}\psi(s).
%\end{equation*}
%Integrating the inequality from $t_{i+1}$ to $t$ and taking into account $\lim_{s\nearrow t}\psi(s)=1$ we get $\psi(t_{i+1})\le C$.
%\begin{equation}
%\psi(t_{i+1})=\int_{M_{i+1}} G(x,t;y,t_{i+1})\,\,d_{t_{i+1}}y\le C.
%\end{equation}
%Moreover, by triangle inequality we have $B_1\subset B_{g(0)}(x_0,r_i-4\sqrt{\frac{t_{i+1}}{\tau}})$ where (APA 3) holds true for $t\in[0,t_{i+1}]$. Hence $\LL(\cdot,t_{i+1})\le 1$ on $\mathcal{B}_{i+1}$ and $\mathcal{I}[\mathcal{B}_{i+1}]\le\psi(t_{i+1})\le C$.
For any $y\in \mathcal{C}_{i+1}$, by the Shrinking Lemma and the assumption on $t_{i+1}$ and $t_{i+2}$ from $\eqref{e: constants conditions}$ we have 
\begin{equation}
d_{t_{i+1}}(x,y)\ge\sqrt[4]{t_{i+2}}-\beta\sqrt{C_3}\sqrt{t_{i+1}}\ge\frac{1}{2}\sqrt[4]{t_{i+2}}\ge \sqrt{t_{i+2}}\ge\sqrt{t-t_{i+1}},
\end{equation}
and hence by Lemma $\ref{calculation: inequity}$ we have 
\begin{equation}\label{quickly}
G(x,t;y,t_{i+1})\le\frac{C}{(t_{i+2})^{\frac{n}{2}}}exp(-\frac{1}{C\sqrt{t_{i+2}}}).
\end{equation}
On the other hand, (APA 2) implies $\LL(y,t_{i+1})\le\frac{C}{t_{i+1}}$, which combining with $\eqref{quickly}$ and Lemma $\ref{calculation: partial integral}$ imply 
\begin{equation}\label{ten}
\mathcal{I}[\mathcal{C}_{i+1}]\le C\,exp(-\frac{1}{C\sqrt{t_{i+2}}})\le C.
\end{equation}
Hence Claim $\ref{rough bound}$ follows by $\eqref{seven}$ and $\eqref{ten}$.
%\begin{equation}\label{I_2}
%\begin{split}
%\mathcal{I}[B_2]
%&\le C\,exp(-\frac{1}{C\sqrt{t_{i+2}}}).
%\end{split}
%\end{equation}
%Combining $\eqref{I_2}$ and $\eqref{I_1}$, we get the assertion of the lemma.
\end{proof}

\textbf{Step 2}\quad 
It remains to convert this upper bound in Lemma $\ref{rough bound}$ to the stronger upper bound as claimed in (APA 3). Using the bound for $\ell$ from Claim $\ref{rough bound}$, we get the following linearization of the evolution equation for $\ell$ on $B_{g(0)}(x_0,r_i-4\sqrt{\frac{t_{i+1}}{\tau}}-\sqrt[4]{t_{i+2}})\times[0,t_{i+2}]$: 
\begin{equation}\label{e: u_better}
\frac{\partial \ell}{\partial t}\le\Delta \ell+\textnormal{scal}\ell+C(n)\ell^2 \le\Delta \ell+\textnormal{scal}\ell+C\ell
\end{equation}
in the barrier sense. Setting $\mathcal{L}(\cdot,t)=e^{-Ct}\ell(\cdot,t)$,
we get
$\frac{\partial}{\partial t}\mathcal{L}\le\Delta \mathcal{L}+\textnormal{scal}\mathcal{L}$ on the same region as above, in the barrier sense.

Hereafter, we choose and fix an arbitrary $(x,t)\in B_{g(0)}(x_0,r_i-4\sqrt{\frac{t_{i+1}}{\tau}}-6\sqrt[4]{t_{i+2}})\times(t_{i+1},t_{i+2}]$. Let $r=\sqrt[4]{t_{i+2}}$ and $R=3r$, then by triangle inequality, $B_{g(0)}(x,R+2r)\subset\subset B_{g(0)}(x_0,r_i-4\sqrt{\frac{t_{i+1}}{\tau}}-\sqrt[4]{t_{i+2}})$, where by Claim $\ref{rough bound}$ we have $\textnormal{Ric}_{g(s)}\ge -K(v_0,n)$ for all $s\in[0,t_{i+2}]$. We now apply Lemma $\ref{l: cutoff}$ to the flow on $B_{g(0)}(x,R+2r)$ during $[0,t_{i+2}]$, and obtain a cut-off function $\phi_{i+1}$ such that 
\begin{equation}\label{e: last1}
B_{g(s)}(x,r)\subset B_{g(0)}(x,2r)\subset \{y\,|\,\phi_{i+1}(y,s)=1\}
\end{equation}
and $supp\,\phi_{i+1}(\cdot,s)\subset B_{g(0)}(x,3r)$, for all $s\in[0,t_{i+2}]$.
Combining with $\eqref{e: last1}$, we find that the supports of $|\nabla\phi_{i+1}|$, $\ps\psi_{i+1}$ and $\Delta\phi_{i+1}$ are all contained in the annulus $A_{2r,3r}(x):=B_{g(0)}(x,3r)-B_{g(0)}(x,2r)$ and we have the following estimates:
\begin{description}
\item[(P1)] $\nabla\phi_{i+1}$ exists a.e. and
$|\nabla\phi_{i+1}|\le =C\,t_{i+2}^{-\frac{n+1}{4}}$; 
\item[(P2)] $\Delta\phi_{i+1}\le  \mu_1:=C\,t_{i+2}^{-\frac{n+1}{2}}$, in the barrier sense;
\item[(P3)] $\psp\phi_{i+1}\le \mu_2:=C\,t_{i+2}^{-\frac{n}{4}}$.
\end{description}
 
In view of $\eqref{e: last1}$ we have $\phi_{i+1}(x,t)=1$ and hence
\begin{equation}
\mathcal{L}(x,t)
%=\mathcal{L}(x,t)\phi_{i+1}(x,t)
=\lim_{s\nearrow t}\int G(x,t;y,s)\mathcal{L}(y,s)\phi_{i+1}(y,s)d_s y.
\end{equation}
The integration domain here and below is always $B_{g(0)}(x,3r)$. In particular, for any integral involving $\nabla \phi_{i+1},\ps\phi_{i+1}$ or $\Delta \phi_{i+1}$, the actual integration domain is contained in $A_{2r,3r}(x)$ since these derivatives vanish at the outside.

Since $G(x,t;\cdot,\cdot)$ is continuous on $B_{g(0)}(x,3r)\times[0,t)$ and smooth on $B_{g(0)}(x,3r)\times(t_j,\min(t_{j+1},t))$ for each $j\le i+1$,
applying Lemma $\ref{smooth of cut-off}$ to $G(x,t;y,s)\phi_{i+1}(y,s)$ and $\LL(y,s)$ and using (P1)-(P3) we obtain
\begin{equation}
\left.\int G\,\phi_{i+1}\LL\right|^{\min(t_{j+1},t)}_{t_j}\le\int_{t_j}^{\min(t_{j+1},t)}\int(G\,\mu_1+G\,\mu_2+2\left\langle\nabla G,\nabla\phi_{i+1}\right\rangle)\LL,
\end{equation}
and hence 
\begin{equation}\label{I+J+K}
\begin{split}
\LL(x,t)%&=\left.\int G\,\phi_{i+1}\LL\right|_{t_1}+ \sum_{j=1}^{i+1} \left.\int G\,\phi_{i+1}\LL\right|^{\min(t_{j+1},t)}_{t_j}\\
&\le\left.\int G\,\phi_{i+1}\LL\right|_{t_1}+\int_{t_1}^{t}\int(G\,\mu_1+G\,\mu_2+2\left\langle\nabla G,\nabla\phi_{i+1}\right\rangle)\LL.
\end{split}
\end{equation}

To estimate the first term in the RHS of $\eqref{I+J+K}$, we first note that on $B_{g(0)}(x,3r)\subset B_{g(0)}(x_0,r_i-4\sqrt{\frac{t_{i+1}}{\tau}}-\sqrt[4]{t_{i+2}})$ we have $\LL(\cdot,t_1)\le 2\alpha_0< 1$ and hence $\Ric_{g(t_1)}\ge-C(n)$ for some dimensional constant $C(n)$. Then applying Lemma $\ref{calculation: whole integral}$ we get 
\begin{equation}\label{latter}
\left.\int G\,\phi_{i+1}\,\LL\right|_{t_1}\le C\cdot 2\alpha_0.
\end{equation}
Then we split the second term in the RHS of $\eqref{I+J+K}$ into two parts:
\begin{equation}\label{haha}
\mathcal{I}=\int_{t_1}^{t}\int (\mu_1+\mu_2)\,G(x,t;y,s)\,\LL(y,s)\,d_sy\,ds,	
\end{equation}
\begin{equation}\label{peach}
\mathcal{J}=2\int_{t_1}^t\int\left\langle\nabla G(x,t;y,s),\nabla\phi_{i+1}(y,s)\right\rangle\mathcal{L}(y,s)\,d_sy\,ds.
\end{equation}
On the one hand, by the Shrinking Lemma, for all $y$ in $A_{2r,3r}(x)$ and $s\in[0,t_{i+2}]$, we have $d_{g(s)}(x,y)\ge d_{g(0)}(x,y)-\sqrt[4]{t_{i+2}}\ge\sqrt[4]{t_{i+2}}$. Thus by Lemma $\ref{calculation: inequity}$ we have 
\begin{equation}\label{hey}
G(x,t;y,s)\le\frac{C}{(t-s)^{\frac{n}{2}}}exp\left(-\frac{\sqrt{t_{i+2}}}{C(t-s)}\right)\le\frac{C}{t_{i+2}^{\frac{n}{2}}}exp\left(-\frac{1}{C\sqrt{t_{i+2}}}\right).
\end{equation}
%Combining this with $\LL(y,s)\le C$ from Lemma $\ref{rough bound}$ we get
%\begin{equation}\label{21}
%(\mu_1+\mu_2)\,G(x,t;y,s)\,\LL(y,s)\le C\,exp\left(-\frac{1}{C\sqrt{t_{i+2}}}\right)
%\end{equation}
On the other hand, let $V(s)$ be the volume of $A_{2r,3r}(x)$ at time $s\in[0,t_{i+2}]$. By Bishop-Gromov comparison, we have $V(0)\le C(n)$. Then we get $V(s)\le C$ by integrating $V'(s)\le C\,V(s)$, which follows from the evolution equation of volume under Ricci flow and Claim $\ref{rough bound}$.
Combining this with $\eqref{hey}$, (P2), (P3) and Claim $\ref{rough bound}$ in $\eqref{haha}$ we can estimate
\begin{equation}\label{I}
 \mathcal{I}\le C\,exp\left(-\frac{1}{C\sqrt{t_{i+2}}}\right).
\end{equation}
Suppose $s \in (t_j,\min(t_{j+1},t))$ for some $j\le i+1$. Since $d_{g(s)}(x,y)\ge \sqrt{t_{i+2}}$ for all $y\in A_{2r,3r}(x)$, applying Claim $\ref{claim: gradient of G}$ of the estimate of $|\nabla G|$, we obtain 
\begin{equation}
\begin{split}
|\nabla G|(x,t;y,s)\le
%\frac{1}{\sqrt{s-t_j}}\frac{C}{t_{i+2}^{\frac{n}{2}}}exp(-\frac{1}{C \sqrt{t_{i+2}}})\le
\frac{C}{\sqrt{s-t_j}}exp(-\frac{1}{C \sqrt{t_{i+2}}}),
\end{split}
\end{equation}
where the constant $C$ depending on $n$ and $v_0$ is uniform for all $j$.
%$s\in\bigcup_{j=1}^i(t_j,t_{j+1})\cup (t_{i+1},t)$. 
%Combining this with $|\nabla \phi_{i+1}|(y,s)\le C t_{i+2}^{-\frac{n+1}{4}}$ and $\LL(y,s)\le C$, we get
Then by Claim $\ref{rough bound}$ we have
\begin{equation}\label{quick}
 |\left\langle\nabla G, \nabla \phi_{i+1}\right\rangle|(y,s)\LL(y,s)\le\frac{C}{\sqrt{s-t_j}}exp(-\frac{1}{C \sqrt{t_{i+2}}}).
\end{equation}
Integrating $\eqref{quick}$ over $A_{2r,3r}(x)\times[t_j,t_{j+1}]$, and then summing over all $j$, we obtain
\begin{equation}\label{J}
\begin{split}
\mathcal{J}&\le Cexp(-\frac{1}{C\sqrt{t_{i+2}}})\sum_{j=1}^{i+1}\sqrt{t_{j+1}-t_j}\\
&=C exp(-\frac{1}{C\sqrt{t_{i+2}}})\sqrt{t_{i+2}(1-\frac{1}{\nu})}(1+\frac{1}{\sqrt{\nu}}+\dots)\\
&=C exp(-\frac{1}{C\sqrt{t_{i+2}}})\frac{\sqrt{(\nu-1)t_{i+2}}}{\sqrt{v}-1}\le C exp(-\frac{1}{C\sqrt{t_{i+2}}}).
\end{split}
\end{equation}
Putting the estimates $\eqref{latter}$, $\eqref{I}$ and $\eqref{J}$ into $\eqref{I+J+K}$, we thus have
\begin{equation}
\mathcal{L}(x,t)\le C exp(-\frac{1}{C\sqrt{t_{i+2}}})+2\alpha_0 C.
\end{equation}
Then there exists positive constant $t(n,v_0,\alpha_0)$, such that the first term can be bounded by $\alpha_0$ when $t_{i+2}\le t(n,v_0,\alpha_0)$, and hence $\LL(x,t)\le \alpha_0(1+2C)$.
Choose $C_4=2(1+2C)$, then $\ell(x,t)\le \mathcal{L}(x,t)e^{Ct}\le 2\mathcal{L}(x,t)\le C_4\alpha_0$. Let $r_{i+1}=r_{i}-4\sqrt{\frac{t_{i+1}}{\tau}}-6\sqrt[4]{t_{i+2}}$, then $\ell(y,s)\le C_4\alpha_0<1$ for any $(y,s)\in B_{g(0)}(x_0,r_{i+1})\times[0,t_{i+2}]$, as claimed in (APA 3). Moreover, by choosing $t(n,v_0,\alpha_0)$ small, we can make sure   
\begin{equation}
r_0-r_{i+1}=\sum\limits_{j=0}^{i} r_{j}-r_{j+1}\le\sum\limits_{j=0}^{i} 4\sqrt{\frac{t_{j+1}}{\tau}}+6\sqrt[4]{t_{j+2}}\le 1.
\end{equation}
So we proved Theorem $\ref{t: theorem1}$.
%\begin{proof}[Proof of Theorem \ref{t: 1.2}]
%By Bishop-Gromov volume comparison, there exists some smaller $v_0>0$ depending only on the original $v_0$, $\varepsilon$, such that for all $x\in B_{g_0}(x_0,1-\frac{\varepsilon}{4})$, we have $Vol_{g_0}B_{g_0}(x,\frac{\varepsilon}{4})\ge v_0$.

%Then we do a rescaling by defining $\tilde{g}=\frac{16}{\varepsilon^2}$, which puts us in exactly the situation of Theorem $\ref{t: theorem1}$, with $s_0=\frac{4}{\varepsilon}\ge 100$ and some new $\alpha_0$ depending on the old $\alpha_0$ and $\varepsilon$. The output of Theorem $\ref{t: theorem1}$ is a Ricci flow on $B_{\tilde{g}_0}(x_0,s_0-2)$ with estimates, and after scaling back, we have a Ricci flow $g(t)$ on $B_{g_0}(x_0,1-\frac{\varepsilon}{2})$, for $t\in[0,\tau]$, where $\tau>0$ depends on $\alpha_0$, $v_0$, $n$ and $\varepsilon$, with $g(0)=g_0$ where defined, and so that
%\begin{equation}
%\begin{dcases}
%|\textnormal{Rm}|_{g(t)}\le\frac{C}{t}\\
%\textnormal{Rm}_{g(t)}+C\alpha_0\textnormal{I}\in\mathcal{C}
%\end{dcases}
%\end{equation}
%on $B_{g_0}(x_0,1-\frac{\varepsilon}{2})$, for all $t\in(0,\tau]$. $C$ depends on $\alpha_0$, $v_0$, $n$ and $\varepsilon$.
%\end{proof}

\end{section}

\begin{section}{Proof of the global existence and bi-H$\ddot{\textnormal{\textbf{o}}}$lder homeomorphism}
  In this section we prove Corollary $\ref{global existence}$ and Corollary $\ref{limit space}$. We need two local curvature estimate lemmas stated below, in both of which the Riemannian manifolds $(M^n,g)$ appearing are not necessarily complete.
  
Lemma $\ref{pseudo}$ is proved by B.L. Chen in \cite[Theorem~3.1]{ChenBL} and Simon in \cite[Theorem~1.3]{speed}.
\begin{lem}\label{pseudo}
 Suppose $(M^n,g(t))$ is a Ricci flow for $t\in[0,T]$, not necessarily complete, with the property that for some $y_0\in M$ and $r>0$, and all $t\in[0,T]$, we have $B_{g(t)}(y_0,r)\subset\subset M$ and 
 \begin{equation}
 |\Rm|_{g(t)}\le\frac{c_0}{t}
 \end{equation}
on $B_{g(t)}(y_0,r)$ for all $t\in(0,T]$ and some $c_0\ge 1$. Then if $|Rm|_{g(0)}\le r^{-2}$ on $B_{g(0)}(y_0,r)$, we must have
\begin{equation}
 |\Rm|_{g(t)}(y_0)\le e^{Cc_0}r^{-2}
\end{equation}
for some $C=C(n)$.
\end{lem}  

Lemma $\ref{Shi}$ is an non-standard version of Shi's derivative estimates. The proof of it can be found in \cite[Lemma~A.4]{incompleteness} and \cite[Theorem~14.16]{PLu}.
\begin{lem}\label{Shi}
 Suppose $(M^n,g(t))$ is a Ricci flow for $t\in [0,T]$, not necessarily complete, with the property that for some $y_0\in M$ and $r>0$, we have $B_{g(0)}(y_0,r)\subset\subset M$ and $|\Rm|_{g(t)}\le r^{-2}$ on $B_{g(0)}(y_0,r)$ for all $t\in[0,T]$, and so that for some $l_0\in\mathbb{N}$ we have initially $|\nabla^{l}\Rm|_{g(0)}\le r^{-2-l}$ on $B_{g(0)}(y_0,r)$ for all $l\in\{1,2,...,l_0\}$. Then there exists $C=C(l_0,n,\frac{T}{r^2})$ such that 
 \begin{equation}
 |\nabla^{l}\Rm|_{g(t)}(y_0)\le Cr^{-2-l}
 \end{equation}
 for each $l\in\{1,2,...,l_0\}$ and all $t\in[0,T]$.
\end{lem}
\begin{proof}[Proof of Corollary \ref{global existence}]
Pick any point $x_0\in M$. We apply Theorem $\ref{t: theorem1}$ with $s_0=k+2$, for each integer $k\ge 2$, giving a Ricci flow $(B_{g_0}(x_0,k),g_k(t))_{t\in[0,\tau]}$ satisfying 
\begin{equation}
\begin{dcases}
 \Rm_{g_k(t)}+C\alpha_0\textnormal{I}\in\mathcal{C}\\
 |\Rm|_{g_k(t)}\le\frac{C}{t}
\end{dcases}
\end{equation}
on $B_{g_0}(x_0,k)$ for all $t\in(0,\tau]$, where $\tau=\tau(n,v_0,\alpha_0)>0$ and $C=C(n,v_0)>0$.

Fix some $r_0>0$. Since $B_{g_0}(x_0,r_0+2)$ is compactly contained in $M$, we have $\sup\,|\Rm|\le\frac{1}{r^2}$ for some $r>0$, where the supremum is taken over $B_{g_0}(x_0,r_0+2)$. For any $y_0\in B_{g_0}(x_0,r_0+1)$, by the Shrinking Lemma we have $B_{g_k(t)}(y_0,\frac{1}{2})\subset B_{g_0}(y_0,1)\subset B_{g_0}(x_0,r_0+2)$ for all $t\in[0,\tau]$, with possibly reducing $\tau$ to a smaller number depending also on $n,\alpha_0$ and $v_0$. Now we can apply Lemma $\ref{pseudo}$ to $g_k(t)$ centered at $y_0$ and get $|\Rm|_{g_k(t)}(y_0)\le K_0=K_0(r,n,C)$. In particular, $K_0$ is independent of $k$. Thus, we have $|\Rm|_{g_k(t)}\le K_0$ on $B_{g_0}(x_0,r_0+1)$ for all $t\in[0,\tau]$.

Then we apply Lemma $\ref{Shi}$ to $g_k(t)$ centered at each $y_0\in B_{g_0}(x_0,r_0)$. The outcome is for each $l\in\mathbb{N}$, there exists $K_1=K_1(n,l,r,\tau)$ such that
\begin{equation}
|\nabla^l\Rm|_{g_k(t)}\le K_1
\end{equation}
on $B_{g_0}(x_0,r_0)$ for all $t\in[0,\tau]$. Again, the $K_1$ is also independent of $k$. Using these derivative estimates in local coordinate charts, by Ascoli-Arzela Lemma we can pass to a subsequence in $k$ and obtain a smooth limit Ricci flow $g(t)$ on $B_{g_0}(x_0,r_0)$ for $t\in[0,\tau]$ with $g(0)=g_0$, which satisfies
\begin{equation}
\begin{dcases}
\Rm_{g(t)}+C\alpha_0\in\mathcal{C}\\
|\Rm|_{g(t)}\le\frac{C}{t}
\end{dcases}
\end{equation}
on $B_{g_0}(x_0,r_0)$, for all $t\in(0,\tau]$.

We repeat this process for larger and larger radii $r_i\rightarrow\infty$, and take a diagonal subsequence to obtain $g_{k_{i}}(t)$ which converges on each $B_{g_0}(x_0,r_0)$. Its limit is a smooth Ricci flow $g(t)$ on the whole of $M$ for $t\in[0,\tau]$ with $g(0)=g_0$.

By the Shrinking Lemma, $B_{g(t)}(x_0,r)\subset B_{g_0}(x_0,r+\beta\sqrt{Ct})\subset\subset M$ for all $t\in(0,\tau]$ and $r>0$. This guarantees that $g(t)$ must be complete for all positive times $t\in(0,\tau]$.
\end{proof}

\begin{proof}[Proof of Corollary \ref{limit space}]
We apply Corollary $\ref{global existence}$ and Lemma $\ref{l: volume}$ to each $(M_i,g_i)$ and obtain a sequence of Ricci flows $(M_i,g_i(t))_{[0,T]}$ with $g_i(0)=g_i$ and $T$ uniform for each $i$, and satisfying the following uniform estimates
\begin{equation}\label{triple}
\begin{dcases}
 \Rm_{g_i(t)}+C\alpha_0\textnormal{I}\in\mathcal{C}\\
 Vol_{g_i(t)}B_{g_i(t)}(x,1)\ge v>0 \\
 |\Rm|_{g_{i}(t)}\le\frac{C}{t}
\end{dcases}
\end{equation}
for all $x\in M_i$ and all $t\in[0,T]$, where constant $C>0$ depends on $n,v_0$, and constants $v,T>0$ depend on $n,v_0,\alpha_0$. And
\begin{equation}\label{distance_0}
 d_{g_i(t_1)}(x,y)-\beta\sqrt{C}(\sqrt{t_2}-\sqrt{t_1})\le d_{g_i(t_2)}(x,y)\le e^{K(t_2-t_1)}d_{g_i(t_1)}(x,y)
\end{equation}
for any $0<t_1\le t_2\le T$ and any $x,y\in M_i$, where $K$ depends on $n,v_0,\alpha_0$. The curvature decay for all positive times provides a uniform bound on the curvature which allows us to apply Hamilton's compactness theorem. We can pass to a subsequence in $i$ so that $(M_i,g_i(t),x_i)\rightarrow (M,g(t),x_{\infty})$ in the Cheeger-Gromov sense, where $(M,g(t))$ is a complete Ricci flow defined over $(0,T]$. $(M,g(t))$ inherits the estimates for curvatures and distance:
\begin{equation}
\begin{dcases}
 \Rm_{g(t)}+C\alpha_0\textnormal{I}\in\mathcal{C}\\
 Vol_{g(t)}B_{g(t)}(x,1)\ge v>0 \\
 |\Rm|_{g(t)}\le\frac{C}{t}
\end{dcases}
\end{equation}
and
\begin{equation}\label{distance}
 d_{g(t_1)}(x,y)-\beta\sqrt{C}(\sqrt{t_2}-\sqrt{t_1})\le d_{g(t_2)}(x,y)\le e^{K(t_2-t_1)}d_{g(t_1)}(x,y)
\end{equation}
for any $0<t_1\le t_2\le T$, and any $x,y\in M$. The inequality $\eqref{distance}$ tells us that $d_{g(t)}$ converges locally uniformly to some metric $d_0$ as $t\searrow 0$. Also, the inequality $\eqref{distance_0}$ tells us that $d_{g_i(t)}$ converges locally uniformly to $d_{g_i}$. So we have $(M_i,d_{g_i},x_i)\rightarrow (M,d_0,x_{\infty})$ in the pointed Gromov-Hausdorff sense. 

By Lemma $\ref{holderinequality}$ we have
\begin{equation}\label{inequality}
 \gamma d_0(x,y)^{1+2(n-1)C}\le d_{g(t)}(x,y)\le e^{Kt}d_0(x,y),
\end{equation}
where $\gamma$ depends only on $n$ and $C$ from $\eqref{triple}$. Then the claim of bi-H$\ddot{\textnormal{o}}$lder homeomorphism follows immediately from this. 
\end{proof}
\end{section}

\begin{section}{Acknowledgement}
The author would like to thank her Ph.D adviser Prof. Richard Bamler for giving numerous valuable advise and his generosity in time.

The author would also like to thank Prof. Peter Topping for interesting discussion.

\end{section}
 
\begin{section}{Reference}

\bibliography{MyCollection}
\end{section}

\end{document}